\documentclass[reqno]{amsart}
\usepackage{amssymb}
\usepackage{mathrsfs}
\usepackage{enumerate}
\usepackage[usenames, dvipsnames]{color}

\usepackage{hyperref}
\usepackage{esint}
\usepackage{verbatim}

\theoremstyle{plain}
\newtheorem{theorem}{Theorem}[section]
\newtheorem{lemma}[theorem]{Lemma}
\newtheorem{corollary}[theorem]{Corollary}
\newtheorem{proposition}[theorem]{Proposition}

\theoremstyle{definition}

\newtheorem{assumption}[theorem]{Assumption}

\theoremstyle{remark}
\newtheorem{remark}[theorem]{Remark}

\newcommand\sfu{{\sf u}}
\newcommand\sfv{{\sf v}}

\newcommand{\Div}{\operatorname{div}}

\numberwithin{equation}{section}

\newcommand{\bC}{\mathbb{C}}

\newcommand{\bR}{\mathbb{R}}
\newcommand{\bH}{\mathbb{H}}
\newcommand{\bZ}{\mathbb{Z}}
\newcommand{\bS}{\mathbb{S}}

\newcommand\cD{\mathcal{D}}

\newcommand\cH{\mathcal{H}}
\newcommand\cL{\mathcal{L}}
\newcommand\cM{\mathcal{M}}

\newcommand\cP{\mathcal{P}}
\newcommand\cQ{\mathcal{Q}}

\newcommand\cS{\mathcal{S}}

\newcommand\cX{\mathcal{X}}

\providecommand{\set}[1]{\{#1\}}

\makeatletter
\def\dashint{\operatorname%
{\,\,\text{\bf-}\kern-.98em\DOTSI\intop\ilimits@\!\!}}
\makeatother

\begin{document}

\title[$L_p$ estimates for elliptic and parabolic equations]{Recent progress in the $L_p$ theory for elliptic and parabolic equations with discontinuous coefficients}

\author[H. Dong]{Hongjie Dong}
\address[H. Dong]{Division of Applied Mathematics, Brown University, 182 George Street, Providence, RI 02912, USA}
\email{Hongjie\_Dong@brown.edu }
%\thanks{H. Dong was partially supported by the NSF under agreement DMS-1600593.}

\begin{abstract}
In this paper, we review some results over the last 10-15 years on elliptic and parabolic equations with discontinuous coefficients. We begin with an approach given by N. V. Krylov to parabolic equations in the whole space with VMO$_x$ coefficients. We then discuss some subsequent development including elliptic and parabolic equations with coefficients which are allowed to be merely measurable in one or two space directions, weighted $L_p$ estimates with Muckenhoupt ($A_p$) weights, non-local elliptic and parabolic equations, as well as fully nonlinear elliptic and parabolic equations.
\end{abstract}

\maketitle

\tableofcontents

\section{Introduction}

The $L_p$-theory of elliptic and parabolic equations with discontinuous coefficients has been studied extensively for more than fifty years. On one hand, when the dimension $d=2$, it is well known that the $W^2_2$ estimate holds for second-order uniformly elliptic operators with general bounded and measurable coefficients; see, for instance, Bernstein \cite{MR1511579} and Talenti \cite{MR0204845}.
On the other hand, a celebrated counterexample in Talenti \cite{MR0201816} and Maugeri et al. \cite{MPS00} indicates that when $d \ge 3$ in general there is no $W^2_2$ estimate for elliptic operators with bounded measurable coefficients even if they are discontinuous only at a single point. Another example due to Ural'tseva \cite{MR0226179} shows the impossibility of the $W^2_p$ estimate when $d\ge 2$ and $p\neq 2$. See also Meyers \cite{Me} and Krylov \cite{MR2667637}. Note that in Ural'tseva's example, the coefficients are continuous except at a single point when $d=2$ or a line when $d=3$. In \cite{MR1612401}, Nadirashvili showed that the weak uniqueness for martingale problems may fail if coefficients are merely measurable and $d \ge 3$. Recently, in \cite{DK14b} it was proved that when $p\neq 2$, there is no $W^2_p$ estimate for elliptic operators with coefficients piecewise constant on each quadrant in $\bR^2$. For divergence form equations, a similar estimate cannot be expected either due to an example by Piccinini and Spagnolo \cite{MR0361422}.
These examples imply that in general there does not exist a solvability theory for uniformly elliptic operators with bounded and measurable coefficients. Thus in the past many efforts have been made to treat particular types of discontinuous coefficients.

Among others, an important class of discontinuous coefficients is the class
of vanishing mean oscillations (VMO). With VMO leading coefficients,
the $L_p$-solvability theorems of second-order linear equations were established in early 1990s \cite{CFL91,CFL93,BC93,DF96} for both divergence and nondivergence form equations. The main technical tool in these papers is the theory of singular integrals, in particular, the Calder\'on--Zygmund theorem and the Coifman--Rochberg--Weiss commutator theorem. However, this approach usually does not allow measurable coefficients because one needs smoothness of the corresponding fundamental solutions.

In 2005, Krylov  \cite{MR2304157} gave a unified kernel-free approach for both divergence and non-divergence linear elliptic and parabolic equations in the whole space, with coefficients which are in the class of VMO with respect to the space variables and are allowed to be merely measurable in the time variable. His proof relies on mean oscillation estimates, and uses the Hardy--Littlewood maximal function theorem and the Fefferman--Stein sharp function theorem. The results in \cite{MR2304157} were shortly extended in \cite{MR2352490} to the case of mixed-norm $L_p-L_q$ spaces, where the mixed norm is defined as $\|f\|_{q,p}=\|f\|_{L_q^t(L_p^x)}$. See also \cite{Kr02, HP97} for earlier results about the mixed-norm estimates for parabolic equations with coefficients independent of $x$, and the monograph \cite{MR2435520}.
Another approach was given earlier by Caffarelli and Peral \cite{MR1486629}, which is based on a level set argument together with a ``crawling of ink spots'' lemma originally due to Krylov and Safonov \cite{KS80,Sa80} in the proof of the celebrated Krylov--Safonov $C^{\alpha}$ estimate for nondivergence form equations with measurable coefficients. See also  Acerbi and Mingione \cite{AM07} for a maximal function free argument applied to the parabolic $p$-Laplace equation, as well as \cite{Iw83,DM93} for earlier work on the elliptic $p$-Laplace equation by using the sharp and maximal functions.
With these approaches, VMO coefficients are treated in a straightforward manner by perturbation argument. Another
advantage is that these approaches are rather flexible: they can be applied to both divergence and non-divergence or even non-local equations with coefficients which are very
irregular in some of the independent variables. Rather than looking for as weak regularity assumptions as possible on coefficients, we emphasize that one of the main points of these work is to put forth new techniques which turn out to be more powerful than singular integrals in several occasions.

In this paper, we will first give an overview of Krylov's approach to parabolic equations in the whole space with VMO$_x$ coefficients mentioned above in Section \ref{sec2}. Then we will review some subsequent important progress in this direction. In Section \ref{sec3}, we discuss divergence and nondivergence form elliptic and parabolic equations with partially and variably partially VMO coefficients. Section \ref{sec4} is about a generalization of the Fefferman--Stein theorem on sharp functions in spaces of homogeneous type and its applications to weighted $L_p$ estimates with Muckenhoupt ($A_p$) weights. In Section \ref{sec5}, we consider two types of non-local elliptic and parabolic equations: equations with non-local derivatives in the space variables and equations with non-local derivatives in the time variable (the Caputo fractional time derivatives). In the last section, we give some examples of fully nonlinear elliptic and parabolic equations.

To keep this review paper within a reasonable length, we do not make any attempt to cover many other important and relevant theories, for instance, the semigroup approach of the maximal regularity, which is from a functional analytic point of view (see \cite{DHP,Weis01, GV17a, GV17b} and the references therein), and recent results about quasilinear equations.

Throughout the paper, we assume that the coefficients are bounded and measurable. In most cases, the lower-order coefficients can be assumed to be unbounded and in certain Lebesgue spaces. See, for example, \cite{LSU68,LU73} and the recent interesting work \cite{KK19,KT20, Kr20b,KRW20}, and the references therein.

\medskip

\noindent {\bf Basic notation.}
Throughout the paper, we always assume that $1 < p, q < \infty$
unless explicitly specified otherwise.
By $N(d,p,\ldots)$ we mean that $N$ is a constant depending only
on the prescribed quantities $d, p,\ldots$.
For a (scalar, vector-valued, or matrix-valued) function $u(t,x)$
in $\bR^{d+1}$, we set
\begin{equation*}
(u)_{\cD} = \frac{1}{|\cD|} \int_{\cD} u(t,x) \, dx \, dt
= \fint_{\cD} u(t,x) \, dx \, dt,
\end{equation*}
where $\cD$ is an open subset in
$$
\bR^{d+1}=\set{(t,x):t\in \bR, x=(x_1,\ldots,x_d)\in \bR^d}
$$
and $|\cD|$ is the
$d+1$-dimensional Lebesgue measure of $\cD$.
For $-\infty\leq S<T\leq \infty$, we %set
%\begin{equation*}
%L_{q,p}((S,T)\times \bR^d)=L_q((S,T),L_p(\bR^d)),
%\end{equation*}
%i.e.,
write $u(t,x)\in L_{q,p}((S,T)\times \bR^d)$ if
\begin{equation*}
\|u\|_{L_{q,p}((S,T)\times \bR^d)}=\left(\int_S^T\left(\int_{\bR^d}
|u(t,x)|^p \,dx\right)^{q/p}\,dt\right)^{1/q}<\infty.
\end{equation*}
%$$
%L_p((S,T)\times \bR^d)= L_{p,p}((S,T)\times \bR^d),
%$$
Define
\begin{align*}
W_{q,p}^{1,2}((S,T)\times \bR^d)&=
\set{u:\,u,u_t,Du,D^2u\in L_{q,p}((S,T)\times \bR^d)},\\
\cH^{1}_{q,p}((S,T)\times \bR^d)&=(1-\Delta)^{1/2}W_{q,p}^{1,2}((S,T)\times \bR^d),\\
\bH^{-1}_{q,p}((S,T)\times \bR^d)&=(1-\Delta)^{1/2}L_{q,p}((S,T)\times \bR^d),
\end{align*}
%$$
%W_p^{1,2}((S,T)\times \bR^d)=
%W_{p,p}^{1,2}((S,T)\times \bR^d),
%$$
%$$
%\cH^{1}_{p}((S,T)\times \bR^d)=
%\cH^{1}_{p,p}((S,T)\times \bR^d),
%$$
%$$
%\bH^{-1}_{p}((S,T)\times \bR^d)=\bH^{-1}_{p,p}((S,T)\times \bR^d).
%$$
and $W_p^{1,2}=W_{p,p}^{1,2}$, $\cH^{1}_{p}=\cH^{1}_{p,p}$, and $\bH^{-1}_{p}=\bH^{-1}_{p,p}$.
For any function $u$ defined in $\Omega\subset \bR^{d+1}$, we denote its H\"{o}lder semi-norm by
$$
[u]_{C^{\alpha/2,\alpha}(\Omega)}=
\sup\Big\{\frac{|u(t,x)-u(s,y)|}{|t-s|^{\alpha/2}+|x-y|^\alpha}\,:\,
\forall\,(t,x)\neq(s,y)\in\Omega\Big\}.
$$
For any $T\in (-\infty,\infty]$, we denote
\begin{equation*}
\bR_T=(-\infty,T), \quad \bR_T^{d+1}=\bR_T\times \bR^d.
\end{equation*}
For $(t,x)=(t,x',x_d)\in \bR\times \bR^d$, let
$$
B_r(x) = \{ y \in \bR^d: |x-y| < r\}, \quad
Q_r(t,x) = (t-r^2,t) \times B_r(x),
$$
$$
B'_r(x') = \{ y' \in \bR^{d-1}: |x'-y'| < r\}, \quad
Q'_r(t,x') = (t-r^2,t) \times B'_r(x').
$$
We set $B_r = B_r(0)$ and $|B_r|$ to be the $d$-dimensional volume of $B_r$. Similarly, we define $B_r'$, $Q_r$, $|B_r'|$, and $|Q_r|$, etc. We use the notation
$$
D_i u = u_{x_i},\quad
D_{ij}u = u_{x_i x_j},
\quad i,j=1,\ldots,d
$$
and the usual summation convention over repeated indices is assumed throughout the paper.

Additional notation will be introduced in later sections when needed.

\section{Krylov's approach to equations with VMO$_x$ coefficients}\label{sec2}

In \cite{MR2304157}, Krylov considered linear parabolic equations in the whole space with VMO$_x$ coefficients which are in the class of VMO with respect to the space variables and are allowed to be merely measurable in the time variable.
Two types of parabolic operators are considered:
\begin{equation}                            \label{eq5.18}
\begin{aligned}
P u &= -u_t + a^{i j}
D_{ij}u + b^{i}
D_i u + c u,\\
\cP u&=-u_t+D_i(a^{ij}D_j u+\hat b^i
u)+b^{i} D_i u+cu
\end{aligned}
\end{equation}
in $\bR^{d+1}$, where all the coefficients are bounded
$$
|a^{ij}(t,x)| \le K,
\quad
|b^{i}(t,x)| \le K,
\quad
|\hat b^{i}(t,x)| \le K,
\quad
|c(t,x)| \le K,
$$
and the $a^{ij}$ satisfy the ellipticity condition
$$
%\sum_{\alpha,\beta=1}^d \sum_{i, j =1}^{m}
a^{ij}\xi_{i} \xi_{j}\ge \delta |\xi|^2
$$
for all $(t,x) \in \bR^{d+1}$, $\xi\in \bR^d$, and for some constants $0 < \delta < 1$ and $K>0$. When the coefficients and solution are time-independent, $P$ and $\cP$ are reduced to the elliptic operators
\begin{equation}                            \label{eq5.19}
\begin{aligned}
L u &=a^{i j} D_{ij}u + b^{i} D_i u + c u,\\
\cL u&=D_i(a^{ij}D_j u+\hat b^i u)+b^{i} D_i u+cu.
\end{aligned}
\end{equation}

The following VMO$_x$ condition was introduced in \cite{MR2304157}. Set
\begin{multline*}
\text{osc}_{\texttt{x}}\left(a^{ij},Q_r(t,x)\right)= r^{-2}|B_r|^{-2}\int_{t-r^2}^{t}\int_{y,z \in B_r(x)}
\left| a^{ij}(s,y) -  a^{ij}(s,z)
\right| \, dy \, dz \, ds
\end{multline*}
and
$$
a_{R}^{\#} = \sup_{(t,x) \in \bR^{d+1}} \sup_{r \le R} \sum_{i,j =1}^{d} \text{osc}_{\texttt{x}}\left(a^{ij},Q_r(t,x)\right).
$$

\begin{assumption}                          \label{assumption20071120}
There is an increasing continuous function $\omega(r)$ defined on $[0,\infty)$
such that $\omega(0) = 0$ and $a_R^{\#} \le \omega(R)$.
\end{assumption}
We note that from the proof, it is easily seen that the assumption above can be relaxed: it is sufficient to assume that $a_R^{\#}$ is sufficiently small for all sufficiently small $R$. Here are the main results of \cite{MR2304157}.

\begin{theorem}[Theorem 2.1 of \cite{MR2304157}]                         \label{thm2.1}
Let $p\in (1, \infty)$, $0 < T < \infty$, and the coefficient matrices of
$P$ satisfy Assumption \ref{assumption20071120}. Then for any $f \in L_{p}((0,T) \times
\bR^d)$, there exists a unique $u \in W_{p}^{1,2}((0,T) \times
\bR^d)$ such that
$$
Pu = f \quad \text{in}\  (0,T) \times \bR^d
$$
and
$u(0,\cdot) = 0$. Furthermore, there is a constant $N$,
depending only on $d$, $p$, $\delta$, $K$, $T$, and the
function $\omega$, such that
$$
\|u \|_{W_{p}^{1,2}((0,T) \times \bR^d)}
\le N \| P u \|_{L_{p}((0,T) \times \bR^d)}
$$
for any $u \in W_{p}^{1,2}((0,T) \times \bR^d)$ satisfying
$u(0,\cdot) = 0$.
\end{theorem}

\begin{theorem}[Theorem 2.4 of \cite{MR2304157}]
                                                \label{thm2.2}
Let $p\in (1,\infty)$, $T\in (0,\infty)$, and the coefficient
matrices of $\cP$ satisfy Assumption \ref{assumption20071120}.
Then for any $f, g_i \in L_{p}((0,T)\times \bR^d)$,
there exists a unique $u \in \cH_{p}^{1}((0,T) \times \bR^d)$
such that
$$
\cP u =  f+ D_i g_i\quad \text{in}\  (0,T) \times
\bR^d
$$
and $u(0,\cdot) = 0$. Furthermore, there is a constant $N$,
depending only on $d$, $p$, $\delta$, $K$, $T$, and the
function $\omega$,  such that
$$
\| u \|_{\cH_{p}^{1}((0,T) \times \bR^d)}
\le N (\| f \|_{L_{p}((0,T) \times \bR^d)}
+\| g \|_{L_{p}((0,T) \times
\bR^d)}).
$$
\end{theorem}

Theorems \ref{thm2.1} and \ref{thm2.2} are derived from the following propositions by using the standard method of continuity and considering $v:=e^{-\lambda u}$.

\begin{proposition}
                            \label{prop2.4}
Let $p\in (1,\infty)$, $T\in (-\infty,\infty]$, and the coefficient
matrices of $P$ satisfy Assumption \ref{assumption20071120}.
Then there exists $\lambda_0\ge 0$
depending only on $d$, $p$, $\delta$, $K$, and the
function $\omega$, and $N>0$ depending only on $d$, $p$, $\delta$, and $K$, such that
for any $\lambda\ge \lambda_0$ and $u\in W^{1,2}_p(\bR^{d+1}_T)$,
\begin{align}
                    \label{eq4.38}
&\|u_t \|_{L_{p}(\bR^{d+1}_T)}
+\|D^2 u \|_{L_{p}(\bR^{d+1}_T)}
+\lambda^{1/2} \|D u \|_{L_{p}(\bR^{d+1}_T)}
+\lambda \|u \|_{L_{p}(\bR^{d+1}_T)}\notag\\
&\le N \| P u-\lambda u \|_{L_{p}(\bR^{d+1}_T)}.
\end{align}
\end{proposition}

\begin{proposition}
%                            \label{prop2.5}
Let $p\in (1,\infty)$, $T\in (-\infty,\infty]$, and the coefficient
matrices of $\cP$ satisfy Assumption \ref{assumption20071120}.
Then there exists $\lambda_0\ge 0$
depending only on $d$, $p$, $\delta$, $K$, and the
function $\omega$, and $N>0$ depending only on $d$, $p$, $\delta$, and $K$, such that
for any $\lambda\ge \lambda_0$ and $u\in \cH^{1}_p(\bR^{d+1}_T)$ satisfying
$$
\cP u-\lambda u=f+ D_i g_i\quad \text{in}\  \bR^{d+1}_T,
$$
we have
$$
\| Du \|_{L_{p}(\bR^{d+1}_T)}+\lambda^{1/2} \| u \|_{L_{p}(\bR^{d+1}_T)}
\le N (\lambda^{-1/2}\| f \|_{L_{p}(\bR^{d+1}_T)}
+\| g \|_{L_{p}(\bR^{d+1}_T)}).
$$
\end{proposition}

Here we only sketch the proof of Proposition \ref{prop2.4} following the idea in \cite{MR2304157} but with a slightly simplified argument used in \cite{MR2771670}. The starting point is the following $L_2$ estimate.

\begin{lemma}
                                \label{lem1.1}
Let $T\in (-\infty,+\infty]$, $a^{ij}=a^{ij}(t)$, and $\lambda\ge 0$. Suppose that $u\in W^{1,2}_2(\bR^{d+1}_T)$.
Then we have
\begin{align*}
&\| u_t\|_{L_2(\mathbb{R}^{d+1}_T)}+
\|D^2 u\|_{L_2(\mathbb{R}^{d+1}_T)}
+\sqrt{\lambda}\|Du\|_{L_2(\mathbb{R}^{d+1}_T)}
+\lambda\|u\|_{L_2(\mathbb{R}^{d+1}_T)}\\
&\leq N(d,\delta)\|Pu-\lambda u\|_{L_2(R^{d+1}_T)}.
\end{align*}
\end{lemma}
\begin{proof}
The lemma is proved by taking the Fourier transform with respect to $x$, solving the resulting ODE with respect to $t$, and then applying the Minkowski inequality and Parseval's identity. Alternatively, the lemma can be proved by using testing the equation with $\lambda u$ and $\Delta u$ and integrating by part, noting that $C_0^{\infty}((-\infty,T]\times \bR^{d})$ is dense in
$ W_2^{1,2}(\mathbb{R}_T^{d+1})$. We omit the details.
\end{proof}

The above lemma also yields the corresponding $L_2$-solvability. We will also use the following interior estimate.

\begin{lemma}
                                    \label{lem2.1}
Suppose that $a^{ij}=a^{ij}(t)$ and $u\in W_2^{1,2}(Q_1)$
satisfies
$$
Pu-\lambda u=0\quad \text{in}\quad Q_1.
$$
Then we have
\begin{equation}
                    \label{eq12.43}
[u]_{C^{1/2,1}(Q_{1/2})}\leq N(d,\delta)\|u\|_{L_2(Q_1)}.
\end{equation}
\end{lemma}
\begin{proof}
We will prove an estimate which is actually stronger than \eqref{eq12.43}.
By mollification in $x$ and then using the Arzel\`a--Ascoli lemma, without loss of generality, we may assume that $u$ is infinitely differentiable in $x$.
For any $1/2\le r<R\le 1$, choose a smooth cutoff function $\eta$ satisfying
$0\leq\eta\leq1$, $\eta=1$ in $Q_{r}$, and $\eta=0$ outside $(-R^2,R^2)\times B_{R}$.
Testing the equation with $u\eta^2$, integrating by
parts, and using  Young's inequality, we get for any $s\in (-R^2,0)$,
\begin{align*}
&\frac 1 2 \int_{B_R}u^2(s,x) \eta^2\,dx+\int_{(-R^2,s)\times B_R}\big(\delta |Du|^2\eta^2+\lambda u^2\eta^2\big)\,dx\,dt\\
&\leq\int_{(-R^2,s)\times B_R}u^2\eta\eta_t-2a^{ij}D_iu\eta \cdot D_j\eta u\,dx\,dt\\
&\leq\int_{(-R^2,s)\times B_R}Nu^2+\frac \delta 2 |Du|^2\eta^2+N|D\eta|^2 u^2\,dx\,dt.
\end{align*}
Taking the supremum with respect to $s\in (-r^2,0)$, we obtain the  Caccioppoli's inequality
$$
\sup_{s\in (-r^2,0)}\int_{B_{r}} u^2\,dx+
\int_{Q_{r}} (|Du|^2+\lambda u^2) \,dx\,dt
\leq N(d,\delta,r,R) \int_{Q_R}u^2\,dx\,dt.
$$
Because $Du$ satisfies the same equation, we have
\begin{align*}
&\sup_{s\in (-r^2,0)}\int_{B_{r}} |Du|^2\,dx+
\int_{Q_{r}} (|D^2u|^2+\lambda |Du|^2) \,dx\,dt\\
&\leq N(d,\delta,r,R) \int_{Q_R}|Du|^2\,dx\,dt.
\end{align*}
By iteration,
\begin{align*}
&\sup_{s\in (-r^2,0)}\int_{B_{r}} |D^ku|^2\,dx+
\int_{Q_{r}} (|D^{k+1}u|^2+\lambda |D^k u|^2) \,dx\,dt\\
&\leq N(d,\delta,r,R,k) \int_{Q_R}|u|^2\,dx\,dt,
\end{align*}
which further implies that
\begin{equation*}
(1+\lambda^2)\sup_{s\in (-r^2,0)}\int_{B_{r}} |D^ku|^2\,dx\leq N(d,\delta,r,R,k) \int_{Q_R}|u|^2\,dx\,dt.
\end{equation*}
Taking $k>d/2$, by the Sobolev embedding theorem,
$$
(1+\lambda)\sup_{Q_{1/2}}|u|\le N\left(\int_{Q_R}|u|^2\,dx\,dt\right)^{1/2}.
$$
Again, because $D^ku$ satisfies the same equation, it follows from the two inequalities above that, for any nonnegative integer $k$,
$$
(1+\lambda)\sup_{Q_{1/2}}|D^k u|\le N\left(\int_{Q_1}|u|^2\,dx\,dt\right)^{1/2}.
$$
Since $u_t=a^{ij}(t)D_{ij}u-\lambda u$, we also have
$$
\sup_{Q_{1/2}}|\partial_tD^ku| \leq N(d,\delta,k)\left(\int_{Q_1}|u|^2\,dx\,dt\right)^{1/2}.
$$
The lemma is proved.
\end{proof}
Observe that \eqref{eq12.43} still holds with $Du$ or $D^2u$ in place of $u$ because the coefficients are time-independent. To proceed further, we define the (parabolic) maximal and sharp function of a function $f\in L_{1,\text{loc}}$ by
$$
\cM  f (t,x) = \sup_{Q \in \cQ, (t,x) \in Q} \fint_{Q} |f(s,y)| \, dy \,ds,
$$
$$
f^{\#}(t,x) = \sup_{Q \in \cQ, (t,x) \in Q} \fint_{Q} |f(s,y) -
(f)_{Q}| \, dy\,ds,
$$
where
$$
\cQ=\{Q_r(z): z=(t,x) \in \bR^{d+1}, r \in (0,\infty)\}.
$$
By the Fefferman--Stein theorem on sharp functions and the  Hardy--Littlewood maximal function theorem, for any $p\in (1,\infty)$ and $f \in L_{p}(\bR^{d+1})$, we have
\begin{equation}
                            \label{eq4.43}
\| f \|_{L_p(\bR^{d+1})} \le N \| f^{\#} \|_{L_{p}(\bR^{d+1})},
\quad
\| \cM  f \|_{L_{p}(\bR^{d+1})} \le N \| f\|_{L_{p}(\bR^{d+1})},
\end{equation}
where $N = N(d,p)$.
It follows from Lemmas \ref{lem1.1}, \ref{lem2.1}, and a decomposition procedure as in \cite{MR2304157}, we have the following mean oscillation estimate.

\begin{lemma}
                                \label{lem2.3}
Let $r\in (0,\infty)$, $\mu\in (0,1/4)$, $\lambda\ge 0$,
$X_0 = (t_0,x_0) \in \bR^{d+1}$, and
$f \in L_{2,\text{loc}}(\bR^{d+1})$.
Assume that $u\in W^{1,2}_{2,\text{loc}}$ satisfies
\begin{equation}
                                \label{eq4.33}
Pu-\lambda u=f
\end{equation}
 in $Q_r(X_0)$. Then we have
\begin{equation}
 			\label{eq5.11}
\left(|U-(U)_{Q_{\mu r}(X_0)}|\right)_{Q_{\mu r}(X_0)}
\le N\mu(|U|^2)_{Q_{r}(X_0)}^{\frac 12}+N\mu^{-1-\frac d 2}(|f|^2)_{Q_{r}(X_0)}^{\frac 12},	
\end{equation}
where $U=|D^2 u|+\lambda^{1/2}|Du|+\lambda |u|$ and $N=N(d,\delta)>0$.
\end{lemma}

From Lemma \ref{lem2.3}, we see that if $u$ satisfies \eqref{eq4.33} in the whole space $\bR^{d+1}$, then we have the pointwise estimate
\begin{equation}
                                \label{eq4.41}
U^\#\le N\mu \cM^{1/2} (U^2)+N\mu^{-1-d/2}\cM^{1/2} (f^2).
\end{equation}
Now we are ready to give the proof of Proposition \ref{prop2.4}.
\begin{proof}[Proof of Proposition \ref{prop2.4}]
For simplicity, we assume $T=\infty$ and $b=c=0$. For the general case, it suffices to take the zero extension of $f$ beyond $T$, and set $\lambda_0$ sufficiently large to absorb lower-order terms.
Here we just give an outline of the proof.

{\bf Step 1:} $p>2$, $a^{ij}=a^{ij}(t)$. In this case, \eqref{eq4.38} is a consequence of \eqref{eq4.41} and \eqref{eq4.43}, by taking a sufficiently small $\mu$. The a priori estimate and the method of continuity then give the $L_p$ solvability in this case.

{\bf Step 2:} $p\in (1,2)$, $a^{ij}=a^{ij}(t)$. In this case, \eqref{eq4.38} is obtained from the previous step and a duality argument, which is possible because $a^{ij}$ are independent of $x$.

{\bf Step 3.} The a priori estimate and the method of continuity then give the $L_p$ solvability when $a^{ij}=a^{ij}(t)$. Now by this as well as a bootstrap/iteration argument, we obtain an estimate analogous to \eqref{eq5.11}:
\begin{equation}
                                \label{eq4.41b}
\left(|U-(U)_{Q_{\mu r}(X_0)}|\right)_{Q_{\mu r}(X_0)}
\le N\mu(U^{p_0})_{Q_{r}(X_0)}^{\frac 1 {p_0}}+N\mu^{-\frac {2+d} {p_0}}(|f|^{p_0})_{Q_{r}(X_0)}^{\frac 1 {p_0}}
\end{equation}
for any $p_0\in (1,\infty)$.

{\bf Step 4:} the general case. We use the argument of freezing the coefficients (with respect to $x$ only). Take $p_0,\nu_1>1$ such that $p_0\nu_1<p$ and let $\nu_2=\nu_1/(\nu_1-1)$.
Assume for the moment that $u$ vanishes whenever $t\notin (-\gamma^2 R^2,0)$ for small constants $\varepsilon,R>0$ to be specified later. For $r\in (0,R)$ and $X\in \bR^{d+1}$, we rewrite the equation into
$$
-u_t+\bar a^{ij}(t)D_{ij}u=\bar f:=(\bar a^{ij}(t)-a^{ij})D_{ij}u,
$$
where
$$
\bar a^{ij}(t)=\fint_{B_r(x)}a^{ij}(t,y)\,dy.
%\begin{cases}
%\fint_{B_r(x)}a^{ij}(t,y)\,dy\quad \text{when}\ r\le R\\
%\fint_{B_R}a^{ij}(t,y)\,dy\quad \text{when}\ r> R
%\end{cases}
$$
It then follows from \eqref{eq4.41b} and H\"older's inequality that
\begin{align*}
&\left(|U-(U)_{Q_{\mu r}(X)}|\right)_{Q_{\mu r}(X)}\\
&\le N\mu(U^{p_0})_{Q_{r}(X}^{\frac 1 {p_0}}
+N\mu^{-\frac {2+d} {p_0}}(|f|^{p_0})_{Q_{r}(X)}^{\frac 1 {p_0}}
+N\mu^{-\frac {2+d} {p_0}}(\omega(R)+\gamma^2)^{\frac 1 {{p_0}\nu_2}}(U^{{p_0}\nu_1})_{Q_{r}(X)}^{\frac 1 {{p_0}\nu_1}},
\end{align*}
which implies the pointwise estimate
\begin{align}
                    \label{eq9.03}
U^\#&\le N\mu \cM^{\frac 1 {p_0}} (U^{p_0})+N\mu^{-\frac {2+d} {p_0}}\cM^{\frac 1 {p_0}} (|f|^{p_0})\notag\\
&\quad +N\mu^{-\frac {2+d} {p_0}}(\omega(R)+\gamma^2)^{\frac 1 {{p_0}\nu_2}}
\cM^{\frac 1 {{p_0}\nu_1}} (U^{{p_0}\nu_1}).
\end{align}
This together with \eqref{eq4.43} yields \eqref{eq4.38} by taking $\mu$ and then $R$ and $\gamma$ sufficiently small. Without the small support condition, we use a partition of unity argument.
\end{proof}

We conclude this section with a few remarks.

\begin{remark}
                    \label{rem2.9}
In \cite{MR2352490}, Theorems \ref{thm2.1} and \ref{thm2.2} were extended to the case of mixed-norm $L_p-L_q$ spaces (with the restriction $q\ge p$ in the nondivergence case). The main idea is to derive from \eqref{eq4.41} a mean oscillation estimate of the $L^x_p$ norms with respect with the time variable only.
\end{remark}

\begin{remark}
The results in this sections were also generalized to second-order parabolic systems in the whole space in \cite{DKim09}, and to higher-order elliptic and parabolic systems in the whole space, on the half space, and on domains in \cite{MR2771670}. In these work, the leading coefficients satisfy the Legendre--Hadamard condition, which is weaker than the strong ellipticity condition. The proof of the boundary estimates is highly nontrivial. We refer the reader to  \cite{MR2771670} for details.
\end{remark}

\section{Equations with partially VMO coefficients}     \label{sec3}

In this section, we discuss elliptic and parabolic equations with partially VMO leading coefficients, i.e., they are merely measurable with respect to an either fixed or varying space direction (and the time variable), and have small mean oscillations with respect to the orthogonal directions on all small scale.

In \cite{Kim07}, Kim obtained the $W^2_p$ estimate and solvability for nondivergence form elliptic equations with leading coefficients which can be discontinuous at finitely many parallel hyperplanes in $\bR^d$. This extends earlier results by Lorenzi \cite{L72a, L72b}, where the coefficients are assumed to be piecewise constant in each half space.

Elliptic equations in non-divergence form with partially VMO coefficients was first considered in \cite{MR2338417}. Let us denote $\Gamma_r(x)=B_r'(x')\times (x_d-r,x_d+r)$ and
$$
\text{osc}_{\texttt{x}'}\left(a^{ij},\Gamma_r(x)\right)
= \fint_{x_d-r}^{x_d+r}
\fint_{B'_r(x')} \big| a^{ij}(y',y_d) - \fint_{B'_r(x')} a^{ij}(z',y_d) \, dz' \big| \, dy' \, dy_d,
$$
and set
\begin{align*}
a^{\texttt{x}',\#}_R &= \sup_{x \in \bR^d} \sup_{r \le R}  \sum_{i,j =1}^{d} \text{osc} \left(a^{ij},\Gamma_r(x)\right).
\end{align*}

\begin{assumption}%                         \label{assumption20080424}
There is an increasing continuous function $\omega(r)$ defined on $[0,\infty)$
such that $\omega(0) = 0$ and $a_{\texttt{x}',R}^{\#} \le \omega(R)$.
\end{assumption}

Recall the operators $P$, $\cP$, $L$, and $\cL$ introduced in \eqref{eq5.18} and \eqref{eq5.19}. The main result of \cite{MR2338417} is the following theorem.

\begin{theorem}[Theorem 2.4 of \cite{MR2338417}]
                            \label{thm3.1}
Let $p\in (2,\infty)$. Then there exists a constant $\lambda_0 \ge 0$, depending
only on $d$, $\delta$, $K$, $p$, and the function $\omega$, such that, for any $\lambda>\lambda_0$ and $f\in L_p(\bR^d)$, there
exists a unique solution $u\in W^2_p(\bR^d)$ satisfying $Lu-\lambda u = f$ in $\bR^d$. Furthermore, there is a constant $N$, depending only on $d$, $\delta$, $K$, and $p$, such that for any $\lambda\ge \lambda_0$ and $u\in W^2_p(\bR^d)$,
$$
\|D^2 u \|_{L_{p}(\bR^{d})}
+\lambda^{1/2} \|D u \|_{L_{p}(\bR^{d})}
+\lambda \|u \|_{L_{p}(\bR^{d})}\le N \| L u-\lambda u \|_{L_{p}(\bR^{d})}.
$$
\end{theorem}
As an application of the above theorem, they also obtained the corresponding result in the half space with either the Dirichlet, Neumann, or oblique
derivative boundary condition, by using an even/odd extension argument. See \cite{MR2300337} for results about nondivergence form parabolic equations, as well as \cite{MR2332574, MR2644213} for interesting extensions to more general coefficients.

The proof of Theorem \ref{thm3.1} is based on the method described in the previous section. The authors first proved the theorem when $p=2$ by using a Fourier method and the maximum principle, and then derived the mean oscillation estimate for $D_{xx'} u$. Finally, the estimate of $D_{dd} u$ can be obtained by using the equation. Note that here one cannot apply the duality argument to get the result for $p\in (1,2)$ because the coefficients are only measurable in $x_d$. This explains why the condition $p>2$ is imposed in the theorem.

In \cite{MR2540989}, Krylov generalized the result in Theorem \ref{thm3.1} to elliptic equations with ``variably'' partially VMO coefficients, i.e., the leading coefficients are assumed to be measurable in one spatial direction and have vanishing mean oscillation in the orthogonal directions in each small ball, with the direction allowed to depend on the ball.
The following generalized Fefferman--Stein sharp function obtained in \cite{MR2540989} plays an essential role in the proof. Let
$$
\bC_n=\{C_n(i_1,\ldots,i_d),i_1,\ldots, i_d\in \bZ\},\quad n\in \bZ
$$
be the filtration of partitions given by cubes, where
\begin{equation*}
C_n(i_1,\ldots, i_d)=[i_1 2^{-n}, (i_1 + 1)2^{-n})\times \cdots \times [i_d 2^{-n}, (i_d + 1)2^{-n}).
\end{equation*}
\begin{theorem}[Theorem 2.7 of \cite{MR2540989}]
                                    \label{generalStein}
Let $p\in (1,\infty)$, $u,v,g\in L_1$. Assume $v\ge |u|$, $g\ge 0$ and for any $n\in \bZ$ and $C\in \bC_n$ there exists a measurable function $u^C$ given on $C$ such that $|u|\le u^C\le v$ on $C$ and
$$
\int_C|u^C-(u^C)_C|\,dx \le \int_C g\,dx.
$$ Then we have
\begin{equation*}
                                    %\label{eq9.44}
\|u\|_{L_p}^p\leq N\|g\|_{L_p}\|v\|_{L_p}^{p-1},
\end{equation*}
where $N=N(d,p)>0$.
\end{theorem}
The proof of this theorem is based on a stopping time argument. It can be readily extended to the parabolic setting and to more general underlying measure spaces with a metric (or quasi-metric). See Section \ref{sec4} for some examples.

For divergence form elliptic and parabolic equations, partially VMO coefficients were later studied in \cite{MR2601069, MR2584743}. In \cite{MR2601069}, elliptic equations were considered in the whole space, on a half space $\{x:x_d>0\}$, or on a bounded Lipschitz domain with a small Lipschitz constant with the Dirichlet and conormal boundary conditions.  The $W^1_p$ estimate and solvability were obtained. In the bounded domain case, the leading coefficients are assumed to be partially VMO in the interior of the domain and VMO near the boundary. Mixed norm estimates were also proved in the whole space and half space cases. In \cite{MR2584743}, the author studied parabolic equations with leading coefficients which are measurable in $(t,x_d)$ and have small mean oscillation in the other variables in small cylinders, except $a^{11}$ which is measurable in $x^d$ and have small mean oscillation in the remaining variables in small cylinders. Let us state one of the main results in \cite{MR2584743}.
\begin{theorem}[Corollary 5.5 of \cite{MR2584743}]
                                        \label{prop1}
Let $p\in (1,\infty)$, $a^{ij}=a^{ij}(x_d)$, $ u \in C_0^\infty$, and $f,g\in  L_p(\bR^{d+1})$. Then there exists a constant $N>0$, depending only
on $d$, $p$, and $\delta$, such that we have
\begin{equation*}
              %                              \label{apriori1}
\lambda \|  u \|_{L_p(\bR^{d+1})}
+ \sqrt\lambda\|Du \|_{L_p(\bR^{d+1})}
\le N\sqrt\lambda\|g\|_{L_p(\bR^{d+1})}+N\|f\|_{L_p(\bR^{d+1})}
\end{equation*}
provided that $\lambda \ge 0$ and $\cP u-\lambda u=f+\Div g$ in $(\bR^{d+1})$. In particular, when $\lambda=0$ and $f\equiv 0$, we have
\begin{equation*}
\|Du \|_{L_p(\bR^{d+1})}
\le N\|g\|_{L_p(\bR^{d+1})}.
\end{equation*}
\end{theorem}

We point out that here the main difficulty is that, since the leading are merely measurable in $x_d$, it is only possible to estimate the sharp function of $D_{x'}u$, not the full gradient $Du$ as in \cite{MR2304157, MR2352490}. Therefore, one need to bound $D_d u$ by $D_{x'}u$. One idea in \cite{MR2601069, MR2584743} is to break the ``symmetry'' of the coordinates so that $t$ and $x_d$ are distinguished from $x'$ by using a delicate re-scaling argument. Another idea is to estimate the sharp function of $a^{dd}D_d u$ instead of $D_d u$, and apply the generalized Fefferman--Stein theorem stated in Theorem \ref{generalStein}.

In \cite{MR2800569}, a new method was developed, which allowed us to treat elliptic and parabolic {\em systems} with variably partially VMO coefficients. In the parabolic setting, the key step in \cite{MR2800569} is to estimate the mean oscillations of $U:=a^{dj} D_j
u$ and $D_iu$, $i = 1, \ldots,d-1,$ instead of the full gradient of
$u$, if the given equation is
$$
u_t-D_i(a^{ij}D_ju) = \Div g.
$$
By using this argument, the authors were able to bypass the scaling argument mentioned above and greatly simplified the proof. Systems with partially VMO coefficients arises from the problems of linearly elastic laminates and composite materials, e.g., in homogenization of layered materials. As an application, a result by Chipot, Kinderlehrer, and Vergara-Caffarelli \cite{CKC} on gradient estimates for elasticity system was improved in \cite{MR2800569}. We refer the reader to \cite{MR2764911} for results on divergence form parabolic systems with more general coefficients, and \cite{MR2835999} for higher-order divergence form elliptic and parabolic systems with variably partially {BMO} coefficients in regular and irregular domains, where the methods in \cite{MR2601069, MR2584743,MR2800569,DK12c} were further developed and a delicate cutoff/reflection argument was introduce (when the usual method of flattening the boundary is unavailable for irregular domains). See also \cite{BW10, BPW13} for similar results proved by using different methods. It is worth noting that the cutoff/reflection argument has been recently applied to study  mixed Dirichlet-conormal boundary conditions in nonsmooth domains.

Coming back to nondivergence form equations, in \cite{MR2896169,MR2833589} the condition $p>2$ in Theorem \ref{thm3.1} and in \cite{MR2540989} was removed. The main idea in these papers is to use the hidden divergence structure of the nondivergence form equations when the coefficients depending only on $x_d$, and then apply the results for divergence form equations in \cite{MR2601069, MR2584743}. To illustrate the idea, below  we give a key result in \cite{MR2896169}, which is stated in the parabolic setting.
\begin{theorem}
						\label{thm3.2}
Let $p\in (1,\infty)$, $T\in (-\infty,\infty]$, and $a^{ij}=a^{ij}(x_d)$. Then for any $u\in W^{1,2}_p(\bR^{d+1}_T)$ and $\lambda\ge 0$, we have
$$
\lambda\|u\|_{L_{p}(\bR^{d+1}_T)}+\sqrt{\lambda}
\|Du\|_{L_{p}(\bR^{d+1}_T)}+\|D^2u\|_{L_{p}(\bR^{d+1}_T)}
+\|u_{t}\|_{L_{p}(\bR^{d+1}_T)}
$$
\begin{equation}
                \label{eq10.30}
\le N\|Pu-\lambda u\|_{L_{p}(\bR^{d+1}_T)},
\end{equation}
where $N=N(d,p,\delta)>0$.
Moreover, for any $f\in L_p(\bR^{d+1}_T)$ and $\lambda>0$ there is a unique $u\in W^{1,2}_p(\bR^{d+1}_T)$ solving $P u-\lambda  u=f$ in $\bR^{d+1}_T$.
\end{theorem}

\begin{proof}
First we assume $T=\infty$. By the method of continuity, it suffices to prove the a priori estimate \eqref{eq10.30} for $u\in C_0^\infty$. Let
\begin{equation}
 				\label{eq2.32}
f=P u-\lambda u.
\end{equation}
We make a change of variables:
$$
y_d=\varphi(x_d):=\int_0^{x_d} \frac 1 {a^{dd}(s)}\,ds,\quad y^j=x^j,\,\,j\ge 2.
$$
It is easy to see that $\varphi$ is a bi-Lipschitz function and
\begin{equation*}
                                            %\label{eq11.51}
\delta \le y_d/x_d\le \delta^{-1},\quad D_{y_d}=a^{dd}(x_d)D_{x_d}.
\end{equation*}
Denote
$$
v(t,y',y_d)=u(t,y',\varphi^{-1}(y_d)),\quad
\tilde a^{ij}(y_d)=a^{ij}(\varphi^{-1}(y_d)),
$$
$$
\tilde f(t,y)=f(t,y',\varphi^{-1}(y_d)).
$$
Define a divergence form operator $\tilde \cP$ by
$$
\tilde \cP v=-v_t+D_{1}\left(\frac 1 {\tilde a^{dd}} D_1 v\right)+\sum_{j=1}^{d-1} D_{j}\left(\frac {\tilde a^{dj}+\tilde a^{jd}}
{\tilde a^{dd}} D_1v
\right)+\sum_{i,j=1}^{d-1} D_{j}(\tilde a^{ij}D_iv).
$$
Clearly, $v$ satisfies in $\bR^{d+1}$
$$
\tilde \cP v-\lambda v=\tilde f,
$$
and $\tilde \cP$ is uniformly nondegenerate.
By Proposition \ref{prop1}, we have
\begin{equation*}
              %                              \label{apriori1}
\lambda \| v \|_{L_p}
+ \sqrt\lambda\|Dv\|_{L_p}
\le N\|\tilde f\|_{L_p}.
\end{equation*}
Therefore,
\begin{equation}
					\label{eq2.25}
\lambda \|u \|_{L_p}
+ \sqrt\lambda\|Du\|_{L_p}
\le N\|f\|_{L_p}.
\end{equation}
Next we estimate $D^2 u$. Notice that for each $k=1,\ldots,d-1$, $D_kv$ satisfies
$$
\tilde \cP (D_k v)-\lambda D_k v=D_k \tilde f.
$$
Again by using Proposition \ref{prop1}, we get
$$
\| D_{yy^k} v \|_{L_p}
\le N\|\tilde f\|_{L_p},
$$
which implies
\begin{equation}
				\label{eq2.57}
\| D_{xx'} u \|_{L_p}
\le N\|f\|_{L_p}.
\end{equation}
Finally, to estimate $D_{dd} u$, we return to the equation in the original coordinates. From \eqref{eq2.32}, we see that $w:=D_d u$ satisfies
\begin{equation*}
						%\label{eq2.55}
-w_t+\Delta_{d-1}w+D_d(a^{dd}D_dw)-\lambda w=D_d f+\sum_{(i,j)\neq (d,d)}D_d
\left((\delta_{ij}-a^{ij})D_{ij}u\right).
\end{equation*}
We use Proposition \ref{prop1} again to get
\begin{equation}
				\label{eq2.36}
\| D_{dd} u \|_{L_p}\le \| Dw \|_{L_p}\le N\| f \|_{L_p}
+N\sum_{(i,j)\neq (d,d)} \| D_{ij}u \|_{L_p}.
\end{equation}
Combining \eqref{eq2.25}, \eqref{eq2.57} and \eqref{eq2.36} yields \eqref{eq10.30} by using
\begin{equation*}
					%\label{eq3.03}
u_t=a^{ij}D_{ij}u-\lambda u-f.
 \end{equation*}

For general $T\in (-\infty,\infty]$, we use the fact that $u=w$ for $t<T$, where $w\in
W_p^{1,2}$ solves
$$
P w-\lambda w=\chi_{t<T}(P u-\lambda u).
$$ The theorem is proved.
\end{proof}
In \cite{MR2833589}, the author extended the results to equations with ``hierarchically'' partially  BMO coefficients by using a scaling argument as in \cite{MR2601069, MR2584743, MR2771670}, but applied to nondivergence form equations.

Given the results mentioned above, it is natural to ask whether one can dispose the regularity of the coefficients in two space dimensions. Unfortunately, because of the counterexamples mentioned at the beginning of the introduction, one cannot expect an $L_p$ theory in this case for either divergence form equations or nondivergence form equations. This means that one can only get the results for a restricted range of $p$. In the divergence case, when $p$ is close to 2, the $W^1_p$ (or $\cH^1_p$) estimate and solvability follows from the usual energy argument and the reverse H\"older's inequality (Gehring's lemma). In the nondivergence case, when $d=2$ and $p=2+\varepsilon$, the $W^2_p$ estimate for elliptic equations with measurable coefficients can be found in Campanato \cite{Ca}. See also \cite{Kr70} for a result about parabolic equations under the condition that $a^{11}+a^{22}$ is a constant or a function depending only on $t$. In the higher dimensional case, we refer the reader to \cite{DKr10}, where $p=2+\varepsilon$ and the leading coefficients are assumed to be measurable in the time variable and two coordinates of space variables, and ``almost'' VMO with respect to the other coordinates.

We finish this section by a remark that for general elliptic (or parabolic) systems in nondivergence form, the $W^2_p$ (or $W^{1,2}_p$) is still an open problem when the coefficients are partially VMO. The argument in \cite{MR2896169,MR2833589} does not work for systems due to the use of the change of variables as shown in the proof of Theorem \ref{thm3.2}. Another interesting problem is the $L_p$ estimate for 1D parabolic equations in non-divergence form with merely measurable coefficients
\begin{equation}
                    \label{1Dpara}
u_t-a(t,x)u_{xx}=f,
\end{equation}
where $a(t,x)\in (\nu,\nu^{-1})$ is a measurable function on
$\bR^2$. For $p=2$, the solvability can be obtained simply by using integration by parts and the method of continuity (see, for example, \cite{MR2300337}). For $p$ sufficiently close to $2$ depending on the ellipticity constant, the answer is also positive, which can be deduced from the main result of \cite{DKr10}. On the other hand, in \cite{Kr16} Krylov presented examples which indicate that when $p \in (1,3/2)\cup (3,\infty)$, in both non-divergence and divergence cases there are equations which are not solvable. The proofs in \cite{Kr16} are based on estimates from below and above for solutions of the Cauchy problem with initial data being an indicator function of a small interval, and then a delicate construction of explicit solutions to the Ornstein--Uhlenbeck equations. The question whether we have the solvability of \eqref{1Dpara} when $p\in[3/2,2)\cup (2,3]$ is still open.

\section{Weighted estimates with Muckenhoupt weights} \label{sec4}

In this section, we showcase some results about solvability and estimates of elliptic and parabolic equations in weighted Sobolev spaces with Muckenhoupt $A_p$ weights. In particular, we will describe how to use a remarkable extrapolation theorem of Rubio de Francia \cite{MR745140} to get rid of the restriction $q\ge p$ in the mixed-norm estimate established in \cite{MR2352490}.

In \cite{DK18}, the authors established a generalized version of the Fefferman--Stein theorem on sharp functions with $A_p$ weights in spaces of homogeneous type with either finite or infinite underlying measure. Let us first fix some notation. Recall that a space $\cX$ of homogeneous type is equipped with a quasi-distance $\rho$ satisfying
\begin{equation*}
				%			\label{eq0709_1}
\rho(x,  y)  \le  K_1(\rho(x,  z)  + \rho(z,  y))
\end{equation*}
for some $K_1\ge 1$ and any $x, y, z \in \cX$, and $\rho(x, y)=0$ if and only if $x = y$,
and a doubling measure $\mu$, i.e., there  exists  a constant $K_2\ge 1$ such that for any $x\in \cX$ and  $r>0$,
\begin{equation*}
%							\label{eq0709_2}
0 < \mu(B_{2r}(x)) \le K_2\mu(B_r(x)) < \infty.
\end{equation*}
We assume that the Lebesgue differentiation theorem is satisfied in $\cX$.
Christ \cite{MR1096400} showed that any such space $\cX$ admits a dyadic decomposition. More precisely,
for each $n\in\bZ$, there is a collection of disjoint open subsets $\bC_n:=\{Q_\alpha^n\,:\,\alpha\in I_n\}$ for some index set $I_n$, which satisfy the following properties
\begin{enumerate}
\item For any $n\in \bZ$, $\mu(\cX\setminus \bigcup_{\alpha}Q_\alpha^n)=0$;
\item For each $n$ and $\alpha\in I_n$, there is a unique $\beta\in I_{n-1}$ such that $Q_\alpha^n\subset Q_\beta^{n-1}$;
\item For each $n$ and $\alpha\in I_n$, $\operatorname{diam}(Q_\alpha^n)\le N_0\delta^n$;
\item Each $Q_\alpha^n$ contains some ball $B_{\varepsilon_0\delta^n}(z_\alpha^n)$
\end{enumerate}
for some constants $\delta\in (0,1)$, $\varepsilon_0>0$, and $N_0$ depending only on $K_1$ and $K_2$.

Denote $\tilde \cX=\bigcap_{n\in \bZ}\bigcup_{\alpha}Q_\alpha^n$. By properties (1) and (2) above, we have
$$
\mu(\cX\setminus \tilde \cX)=0,\quad \tilde \cX=\lim_{n\nearrow \infty}\bigcup_{\alpha}Q_\alpha^n.
$$
By properties (2), (3), and (4), we have
\begin{equation}
                                    \label{eq4.57}
\mu(Q_\beta^{n-1})\le N_1\mu(Q_\alpha^{n})
\end{equation}
for any $n$, $\alpha\in I_n$, $\beta\in I_{n-1}$ such that $Q_\alpha^n\subset Q_\beta^{n-1}$.

For a function $f\in L_{1,\text{loc}}(\cX,\mu)$ and $n\in \bZ$, we set
$$
f_{|n}(x)=\fint_{Q_\alpha^n}f(y)\,\mu(dy),
$$
where $x\in Q_\alpha^n \in \bC_n$.
For $x\in \tilde \cX$, we define the (dyadic) maximal function and sharp function of $f$ by
\begin{align*}
\cM_{\text{dy}} f(x)&=\sup_{n<\infty}|f|_{|n}(x),\\
f_{\text{dy}}^{\#}(x)&=\sup_{n<\infty}\fint_{Q_\alpha^n\ni x}|f(y)-f_{|n}(x)|\,\mu(dy).
\end{align*}
%Based on this we can define the dyadic maximal function $\cM_{\text{dy}} f$ and the dyadic sharp function $f^{\#}$ for any $f\in L_{1,\text{loc}}(\cX,\mu)$.
It is easily seen that
\begin{equation*}
                %    \label{eq8.52}
\cM_{\text{dy}} f(x)\le N \cM f(x)\quad \text{and}\quad
f_{\text{dy}}^{\#}(x)\le Nf^{\#}(x)
\quad \mu\text{-a.e.},
\end{equation*}

For any $p\in (1,\infty)$, let $A_p=A_p(\mu)= A_p(\cX, d \mu)$ be the set of all nonnegative functions $w$ on $(\cX,\rho,\mu)$ such that
$$
[w]_{A_p}:=\sup_{x_0\in \cX,r>0}\left(\fint_{B_r(x_0)}w(x)\,d\mu\right)
\left(\fint_{B_r(x_0)}\big(w(x)\big)^{-\frac 1 {p-1}}\,d\mu\right)^{p-1}<\infty.
$$
By H\"older's inequality, we have
\begin{equation*}
                %                \label{eq7.53}
A_p\subset A_q,\quad
1 \le
[w]_{A_q}\le [w]_{A_p},
\quad 1<p<q<\infty.
\end{equation*}
Denote
$$
A_\infty=\bigcup_{p=2}^\infty A_p.
$$
We write $f \in L_p(\cX, w \, d\mu)$ or simply $f \in L_p(w \, d \mu)$ if
$$
\int_\cX |f|^p w \, \mu(dx) = \int_\cX |f|^p w \, d\mu < \infty.
$$
We use $w(\cdot)$ to denote the measure $w(dx)=w\mu(dx)$, i.e., for $A \subset \cX$,
\begin{equation*}
%							\label{eq0316_1}
w(A) = \int_A w(x) \, \mu(dx).
\end{equation*}

The following Hardy--Littlewood maximal function theorem with $A_p$ weights was obtained in \cite{MR0740173} by Aimar and Mac{\'{\i}}as.
\begin{theorem}
                                    \label{thm4.1}
Let $p\in (1,\infty)$, $w\in A_p$. Then for any $f\in L_p(w\,d\mu)$, we have
$$
\|\cM f\|_{L_p(w\,d\mu)}\le N\|f\|_{L_p(w\,d\mu)},
$$
where $N=N(K_1,K_2,p,[w]_{A_p})>0$.
If $[w]_{A_p}\le K_0$ for some $K_0\ge 1$, then one can choose $N$ depending only on $K_1$, $K_2$, $p$, and $K_0$.
\end{theorem}
Note that in Theorem \ref{thm4.1}, $\mu(\cX)$ can be either finite or infinite, and $\cX$ is allowed to be a bounded space with respect to $\rho$. However, for the Fefferman--Stein theorem, to bound the $L_p$ norm of $f$ by that of $f^\#$ it is crucial that the underlying measure is infinite because otherwise we have nonzero constant functions as a simple counterexample. One of the main results of \cite{DK18} is the following theorem
on sharp functions with $A_\infty$ weights.

\begin{theorem}[Theorem 2.3 of \cite{DK18}]
                                    \label{thm4.2}
Let $p,q \in (1, \infty)$, $\varepsilon>0$, $K_0 \ge 1$, $w\in A_q$, $[w]_{A_q}\le K_0$, and $f\in L_p(w\,d\mu)\cap L_{1,\text{loc}}(d\mu)$.

\begin{enumerate}
\item[(i)] When $\mu(\cX)<\infty$, we have
\begin{equation*}
                %            \label{eq1.51}
\|f\|_{L_p(w\,d\mu)}\le N\|f_{\operatorname{dy}}^{\#}\|_{L_p(w\,d\mu)}+
N(\mu(\cX))^{-1}\big(w(\operatorname{supp} f)\big)^{\frac 1 p}\|f\|_{L_1(\mu)}.
\end{equation*}
If in addition we assume that $\operatorname{supp} f \subset B_{r_0}(x_0)$ and $\mu(B_{r_0}(x_0))\le \varepsilon\mu(\cX)$ for some $r_0\in (0,\infty)$ and $x_0\in \cX$, then
\begin{equation*}
                %                        \label{eq3.41}
\|f\|_{L_p(w\,d\mu)}\le N\|f_{\operatorname{dy}}^{\#}\|_{L_p(w\,d\mu)}+
N\varepsilon \|f\|_{L_p(w\,d\mu)}.
\end{equation*}
Here $N=N(K_1,K_2,p,q,K_0)>0$. In particular,  when $\varepsilon$ is sufficiently small depending on $K_1$, $K_2$, $p$, $q$, $K_0$, it holds that
\begin{equation}
                            \label{eq1.54}
\|f\|_{L_p(w\,d\mu)}\le N\|f_{\operatorname{dy}}^{\#}\|_{L_p(w\,d\mu)}.
\end{equation}

\item[(ii)] When $\mu(\cX)=\infty$, \eqref{eq1.54} always holds.
\end{enumerate}
\end{theorem}

A similar result was obtained in \cite[Theorem 4.2]{MR2033231} in the sense that both theorems deal with spaces of homogeneous type with a finite or infinite underlying measure. In the case of a finite underlying measure, the above theorem is more quantitative than Theorem 4.2 in \cite{MR2033231} and is stated in such a way that one can control the weighted $L_p$ or $L_{p,q}$-norm of a function $f$ by that of $f^{\#}_{\text{dy}}$ if the support of $f$ is sufficiently small. This is particularly useful when equations in bounded domains are considered. In fact, a more general result was proved, which also improves Theorem \ref{generalStein} and is handy when one treats equations with variably partially VMO coefficients.

\begin{theorem}[Theorem 2.4 of \cite{DK18}]							\label{thm4.3}
Let $p,q \in (1, \infty)$, $\varepsilon>0$, $K_0 \ge 1$, $w\in A_q$, and $[w]_{A_q}\le K_0$.
Suppose that
$$
f, g, v\in L_p(w\,d\mu)\cap L_{1,\text{loc}}(d\mu),\quad |f| \le v,
$$
and for each $n \in \bZ$ and $Q \in \bC_n$,
there exists a measurable function $f^Q$ on $Q$
such that $|f| \le f^Q \le v$ on $Q$ and
\begin{equation*}%							 \label{eq10.53}
\fint_Q |f^Q(x) - \left(f^Q\right)_Q| \,\mu(dx)
\le g(y)\quad \forall \,y\in Q.
\end{equation*}

\begin{enumerate}
\item[(i)] When $\mu(\cX)<\infty$, we have
\begin{equation*}
%                            \label{eq1.51b}
\|f\|^p_{L_p(w\,d\mu)}\le N\|g\|^\beta _{L_p(w\,d\mu)}\|v \|_{L_p(w\,d\mu)}^{p-\beta}+
N(\mu(\cX))^{-p} \omega (\operatorname{supp} v)\|v\|^p_{L_1(\mu)}.
\end{equation*}
If in addition we assume that $\operatorname{supp} v \subset B_{r_0}(x_0)$ and $\mu(B_{r_0}(x_0))\le \varepsilon\mu(\cX)$ for some $r_0\in (0,\infty]$ and $x_0\in \cX$, then
\begin{equation}
                                \label{eq12.09}
\| f \|_{L_p(w\,d\mu)}^p
\le N\|g\|^\beta _{L_p(w\,d\mu)}\|v \|_{L_p(w\,d\mu)}^{p-\beta}
+N\varepsilon^p\|v \|_{L_p(w\,d\mu)}^{p}.
\end{equation}
Here $(\beta, N) =  (\beta,N)(K_1,K_2,p,q,K_0)>0$.

\item[(ii)] When $\mu(\cX)=\infty$, \eqref{eq12.09} holds with $\varepsilon=0$.
\end{enumerate}
\end{theorem}
The proofs of Theorems \ref{thm4.2} and \ref{thm4.3} use a usual stopping time argument, but with a threshold
\begin{equation*}
				%			\label{eq0401_1}
\lambda_0 =
\left\{\begin{aligned}
2 N_1 \|f\|_{L_1(\mu)} \left(\mu(\cX)\right)^{-1}
\quad &\text{if} \quad \mu(\cX) < \infty,
\\
0 \quad &\text{if} \quad \mu(\cX) = \infty,
\end{aligned}
\right.
\end{equation*}
where $N_1$ is a constant in \eqref{eq4.57}. For a nonnegative $f\in L_{1,\text{loc}}$, we define the stopping time:
$$
\tau(x) = \inf_{n \in \bZ} \{ n : f_{|n}(x) > \alpha \lambda \},
$$
where $\lambda>\lambda_0$. This allows us to estimate the measure of the super level set $w\{x\,:\,|f(x)|\ge \lambda\}$ by using the Calder\'on--Zygmund decomposition when $\lambda>\lambda_0$. When $\lambda\in (0,\lambda_0]$, we estimate the above measure simply by $w(\text{supp}\ f)$. We refer the reader to \cite{DK18} for details.

From Theorems \ref{thm4.1}-\ref{thm4.3}, the following mixed-norm results were derived by using the extrapolation theorem of Rubio de Francia, which roughly speaking read that if a pair of function $(f,g)$ satisfies
$$
\|f\|_{L_{p}(w\ d\mu)}\le N\|f\|_{L_{p}(w\ d\mu)}
$$
for some $p\in (1,\infty)$ and any $w\in A_p$. Then it satisfies
$$
\|f\|_{L_{q}(w\ d\mu)}\le N\|f\|_{L_{q}(w\ d\mu)}
$$
for {\em any} $q\in (1,\infty)$ and any $w\in A_q$. See \cite{MR745140} as well as \cite[Theorem 2.5]{DK18} for the precise statement.

We write $x=(x',x'')$, and let $(\cX',\rho_1,\mu_1)$ and $(\cX'',\rho_2,\mu_2)$ be two spaces of homogeneous type. Define $\mu$ to be the completion of the product measure on $\cX'\times \cX''$ and
$$
\rho((x',x''), (y',y''))=\rho_1(x',y')+\rho_2(x'',y'')
$$
be a quasi-metric on $\cX'\times \cX''$. Let $\cX\subset \cX'\times \cX''$ such that $(\cX,\rho|_{\cX\times \cX},\mu|_\cX)$ is of homogeneous type. Assume that for any $p\in (1,\infty)$, $w_1=w_1(x')\in A_p(\mu_1)$, and $w_2=w_2(x'')\in A_p(\mu_2)$,
$w=w(x):=w_1(x')w_2(x'')$ is an $A_p$ weight on $(\cX,\rho|_{\cX\times \cX},\mu|_\cX)$ and
$$
[w]_{A_p}\le N([w_1]_{A_p},[w_2]_{A_p}).
$$
We remark that this condition is always satisfied, for example, if $\cX$ satisfies the interior measure condition in $\cX'\times \cX''$.

For any $p,q\in (1,\infty)$ and weights $w_1=w_1(x')$ and $w_2=w_2(x'')$, we define the weighted mixed norm on $\cX$ by
\begin{equation*}
			%				\label{eq0806_02}
\begin{aligned}
&\|f\|_{L_{p,q,w}(d\mu)}=\|f\|_{L_{p,q}(w\,d\mu)}\\
&:=\|f\|_{L_{p,q}(\cX,w\,d\mu)}=\left(\int_{\cX''}\big(\int_{\cX'} |f|^p I_{\cX} w_1(x')\,\mu_1(dx')\big)^{q/p}w_2(x'')\,\mu_2(dx'')\right)^{1/q}.
\end{aligned}
\end{equation*}

\begin{corollary}[Corollary 2.6 of \cite{DK18}]
                                    \label{cor1}
Let $p,q\in (1,\infty)$, $K_0 \ge 1$, $w_1=w_1(x')\in A_p(\mu_1)$, $w_2=w_2(x'')\in A_q(\mu_2)$,
$[w_1]_{A_p}\le K_0$, $[w_2]_{A_q}\le K_0$, and $w=w(x)=w_1(x')w_2(x'')$. Then for any $f\in L_{p,q}(w\,d\mu)$, we have
$$
\|\cM f\|_{L_{p,q}(w\,d\mu)}\le N\|f\|_{L_{p,q}(w\,d\mu)},
$$
where $N=N(K_1,K_2,p,q,K_0)>0$.
\end{corollary}

\begin{corollary}[Corollary 2.8 of \cite{DK18}]
                                    \label{cor3}
Let $p,q,p',q'\in (1,\infty)$, $K_0 \ge 1$, $w_1=w_1(x')\in A_{p'}(\mu_1)$, $w_2=w_2(x'')\in A_{q'}(\mu_2)$,
$[w_1]_{A_{p'}}\le K_0$, $[w_2]_{A_{q'}}\le K_0$,
$w=w(x):=w_1(x')w_2(x'')$, and $f, g \in L_{p,q}(w\,d\mu)$.
Suppose that either $\mu(\cX)=\infty$ or $\operatorname{supp} f\subset B_{r_0}(x_0)$ and $\mu(B_{r_0}(x_0))\le \varepsilon\mu(\cX)$ for some $r_0\in (0,\infty]$ and $x_0\in \cX$, where $\varepsilon>0$ is a constant depending on $K_1$, $K_2$, $p$, $p'$, $q$, $q'$, and $K_0$.
Moreover, for each $n \in \bZ$ and $Q \in \bC_n$,
there exists a measurable function $f^Q$ on $Q$
such that $|f| \le f^Q \le K_3|f|$ on $Q$ for some constant $K_3>0$
and
\begin{equation*}	
\fint_Q |f^Q(x) - \left(f^Q\right)_Q| \,\mu(dx)
\le g(y)\quad \forall \,y\in Q.
\end{equation*}
Then we have
\begin{equation*}
\| f \|_{L_{p,q}(w\,d\mu)}
\le N\|g\|_{L_{p,q}(w\,d\mu)},
\end{equation*}
where $N=N(K_1,K_2,K_3,p,q,p',q',K_0)>0$.
\end{corollary}

Below we give an application of the above results to second-order parabolic equations in nondivergence form equations with VMO$_x$ coefficients. Before that, we recall the self-improving property of $A_p$ weights: for any $w\in A_p(\bR^d),p\in (1,\infty)$ satisfying $[w]_{A_p}\le K_0$, there exists $\varepsilon=\varepsilon(d,p,[w]_{A_p})\in (0,p-1)$ such that
$$
w\in A_{p-\varepsilon}\quad \text{and}\quad
[w]_{A_{p-\varepsilon}}\le N(d,p,[w]_{A_p}).
$$
See, for example, \cite[Theorem 3.2]{MS1981}. Now for any $p,q\in (1,\infty)$, $w_1\in A_p(\bR^d)$ with $[w_1]_{A_p}\le K_0$, and $w_2\in A_q(\bR)$ with $[w_2]_{A_q}\le K_0$, we choose
$$
\varepsilon_1 = \varepsilon_1(d, p, K_0),
\quad
\varepsilon_2 = \varepsilon_2(q, K_0)
$$
such that $p-\varepsilon_1 > 1$, $q - \varepsilon_2 >1$ and
$$
w_1 \in A_{p-\varepsilon_1}(\bR^{d}),
\quad
w_2 \in A_{q-\varepsilon_2}(\bR).
$$
Find $p_0, \nu_1 \in (1,\infty)$ satisfying
\begin{equation*}
							%\label{eq0610_07}
p_0 \nu_1 = \min\left\{ \frac{p}{p-\varepsilon_1}, \frac{q}{q-\varepsilon_2} \right\}> 1.
\end{equation*}
Note that
\begin{equation*}
				%			\label{eq0605_13}
\begin{aligned}
w_1 &\in A_{p-\varepsilon_1}(\bR^{d}) \subset A_{p/(p_0\nu_1)}(\bR^{d}) \subset A_{p/p_0}(\bR^{d}),
\\
w_2 &\in A_{q-\varepsilon_2}(\bR) \subset A_{q/(p_0\nu_1)}(\bR) \subset A_{q/o_0}(\bR).
\end{aligned}
\end{equation*}
By using \eqref{eq9.03} and Corollaries \ref{cor1}-\ref{cor3}, we immediately get the following theorem, which is a special case of \cite[Theorem 5.2]{DK18} and improves one of the main results in \cite{MR2352490} by dropping the restriction $q\ge p$. See Remark \ref{rem2.9}.

\begin{theorem}
Let $p,q\in (1,\infty)$, $w_1$ and $w_2$ satisfy the assumptions above, $w=w_1(x)w_2(t)$, $T\in (-\infty,\infty]$, and the coefficients of $P$ satisfy Assumption \ref{assumption20071120}.
Then there exists $\lambda_0\ge 0$
depending only on $d$, $p$, $q$, $\delta$, $K$, $K_0$, and the
function $\omega$, and $N>0$ depending only on $d$, $p$, $q$, $\delta$, $K$, $K_0$, such that for any $\lambda\ge \lambda_0$ and $u\in W^{1,2}_{p,q,w}(\bR^{d+1}_T)$,
\begin{align*}
             %       \label{eq4.38b}
\|u_t \|+\|D^2 u \|+\lambda^{1/2} \|D u \|
+\lambda \|u \|\le N \| P u-\lambda u \|,
\end{align*}
where $\|\cdot\|=\|\cdot\|_{L_{p,q,w}(\bR^{d+1}_T)}$.
\end{theorem}

We refer the reader to \cite{DK18} for various applications to weighted mixed-norm estimate for higher-order nondivergence form systems in the whole space or on the half space with VMO$_x$ coefficients, second-order nondivergence form equations with partially VMO coefficients, and higher-order divergence form systems with partially VMO coefficients on bounded or unbounded Reifenberg flat domains. From these a priori estimates, one can also derive the corresponding solvability results. See Section 8 of \cite{DK18}.

The argument developed in \cite{DK18} has later been applied in \cite{DG19,DG19b} to study higher-order elliptic and parabolic equations with VMO coefficients and with general boundary condition satisfying the Lopatinskiĭ--Shapiro condition, as well as in \cite{CK19} where weighted mixed  norm estimates were obtained for higher-order elliptic and parabolic systems with VMO$_x$ coefficients on Reifenberg flat domains.
See also \cite{DKr18} for some extensions with applications to fully nonlinear elliptic and parabolic equations, \cite{DP18,DKP19} for applications to time-dependent Stokes systems with VMO$_x$ coefficients, as well as \cite{HKP19} for applications to time-nonlocal equations.
We also refer the reader to \cite{Kr19} for a recent survey on the extrapolation theory as well as \cite{Kr20} for an application to the weak uniqueness of solutions of the stochastic Itô equations when the leading coefficients are assumed to be measurable in the time variable and two space variables, and be almost in VMO with respect to the other coordinates.

It would be interesting to see whether the results in \cite{DG19,DG19b} can be extended to the VMO$_x$ case.

To conclude this section, we mention that there are many important work in the literature when the weight is certain power of the distance to the boundary of the domain, which in particular have applications to SPDEs. See \cite{Kr99, Kr01, KK04, KN09, MR3318165,DKZ14}, and the references there in.

\section{Nonlocal elliptic and parabolic equations} \label{sec5}

In this section, we discuss two types of non-local elliptic and parabolic equations.
Non-local equations such as integro-differential equations with jump L\'evy processes have
received increasing attention. See recent work of Caffarelli, Silvestre, and many other people. These equations arise from models in physics, engineering, and finance that involve long-range interactions. An example is the
following equation
\begin{align}
                                \label{nonloc}
&Lu(x)-\lambda u(x)\notag\\
&:=\int_{\bR^d} \left(u(x+y) - u(x)-y\cdot \nabla u(x)\chi^{(\sigma)}(y)\right)K(x,y)\, d y-\lambda u(x)=f(x),
\end{align}
where
$$
\chi^{(\sigma)}\equiv 0\quad \text{for}\,\, \sigma\in (0,1),
\quad \chi^{(1)}=1_{y\in B_1},\quad
\chi^{(\sigma)}\equiv 1 \quad \text{for}\,\,\sigma\in (1,2),
$$
and $K(x,y)$ is a positive kernel which has certain
lower and upper bounds. %and satisfies some additional regularity conditions.
A typical case is the fractional Laplace operator $-(-\Delta)^\gamma, \gamma\in (0,1)$ which can be written as the principal value of the integral above with
$K(x,y)=C_{d,\gamma}|y|^{-d-2\gamma}$ for a constant
$C_{d,\gamma}>0$. In this case, the classical theory for pseudo-differential operators shows that, for any $\lambda > 0$ and $f \in L_p(\bR^d)$, $1<p<\infty$, there exists a unique solution $u$ in the Bessel potential space
$$
H_p^{\sigma}(\bR^d):
= \{ u \in L_p(\bR^d): (1-\Delta)^{\sigma/2} u \in L_p(\bR^d) \}
$$
to \eqref{nonloc} satisfying
$$
\|u\|_{H_p^\sigma(\bR^d)}
\le N(d,\sigma,\lambda, p) \|f\|_{L_p(\bR^d)}.
$$
In general, if the symbol of the operator is sufficiently smooth and its derivatives satisfy appropriate decays,
the aforementioned $L_p$-solvability is classical following from the Fourier multiplier theorems. On the other hand, the $L_p$-solvability is also available if
\begin{equation}
                                \label{eq8.31}
K(x,y)=K(y)=a(y)/|y|^{d+\sigma}
\end{equation}
and $a(y)$ is homogeneous of order zero and sufficiently smooth. See \cite{Ko84,MP92}.

In \cite{DK12a}, the authors extended this type of $L_p$-solvability to the equation \eqref{nonloc} %\footnote{One can also consider the equation \eqref{elliptic} with $\chi^{(\sigma)} = 1_{y \in B_1}$ for all $\sigma \in (0,2)$. For a discussion about this case, see Remark \ref{remark1024}.}
when the kernel $K$ is translation invariant with respect to $x$, i.e., $K(x,y)=K(y)$, merely measurable in $y$, and satisfies the ellipticity condition
\begin{equation}
                                    \label{eq7.07}
(2-\sigma)\frac{\nu}{|y|^{d+\sigma}}
\le K(y) \le (2-\sigma)\frac{\nu^{-1}}{|y|^{d+\sigma}}, \quad \nu\in (0,1),
\end{equation}
where $0 < \sigma < 2$ is a constant.
By using a purely analytic method, the authors proved the continuity of the non-local operator $L$ from the Bessel potential space $H^\sigma_p$ to $L_p$, and the unique strong solvability of the corresponding non-local elliptic equations in $L_p$ spaces.
\begin{theorem}[Theorem 2.1 of \cite{DK12a}]
								%\label{mainth01}
Let $1< p < \infty$, $\lambda \ge 0$, and $0< \sigma < 2$. Assume that $K=K(y)$ satisfies \eqref{eq7.07} and, if $\sigma=1$, $K$ also satisfies the additional cancellation condition
$$
\int_{\partial B_r}yK(y)\, dS_r(y)=
0,\quad \forall r\in (0,\infty).
$$
Then $L$ defined in \eqref{nonloc} is a continuous operator from $H_p^\sigma$ to $L_p$.
For $u \in H_p^{\sigma}$ and $f \in L_p$ satisfying
\begin{equation}
								\label{eq1220}
Lu - \lambda u = f	\quad\text{in}\,\,\bR^d,
\end{equation}
we have
\begin{equation*}
                %                \label{eq24.09.53}
\|u\|_{\dot H_p^\sigma} + \sqrt{\lambda}\|u\|_{\dot{H}^{\sigma/2}_p}
+ \lambda \|u\|_{L_p}
\le N \|f\|_{L_p},
\end{equation*}
where $N=N(d,\nu,\Lambda,\sigma,p)$.
Moreover, for any $\lambda > 0$ and $f \in L_p$, there exists a unique strong solution $u \in H_p^{\sigma}$ of \eqref{eq1220}.
\end{theorem}
As a byproduct, interior $L_p$-estimates were also obtained.
The proof is built on mean oscillation estimates mentioned above. The key step in establishing the mean oscillation estimates of solutions is based on a $C^{\alpha}$ estimate for the non-local equation, for which some ideas from Barles et al. \cite{Ba11} were used.
Compared to previous results, the novelty of the results is that the function $a$ in \eqref{eq8.31} is not necessarily to be homogeneous, regular, or symmetric. An application of the above result is the uniqueness for the martingale problem associated to the operator $L$ upon using the embedding $C^0\subset H^\sigma_p$ for $p>d/\sigma$. See, for instance, \cite{Ko84}. %In a subsequent paper \cite{DK13}, we considered Schauder estimates for the same class of operators with either Dini or H\"older continuous data. As an application, we proved that the operators give isomorphisms between the Lipschitz--Zygmund spaces $\Lambda^{\alpha+\sigma}$ and $\Lambda^\alpha$ for any $\alpha>0$. Several local estimates and an extension to operators with more general $x$-dependent kernels are also discussed.
Similar results for parabolic equations were later obtained in \cite{MP14}.

It is however not clear to us whether the results in \cite{DK12a} still hold if the kernel $K(x,y)$ can be written as $a(x,y)/|y|^{d+\sigma}$ and $a$ is uniformly continuous with respect to $x$. Unlike the second-order case, this extension seems to be highly nontrivial. For instance, it is not even clear whether the operator $L$ is bounded from $H^\sigma_p$ to $L_p$ under the continuity condition on $a$. This is also highly in contrast with the Schauder theory for the same equation (see, for instance, \cite{DK13}), where for the $C^{\sigma+\alpha}$ estimate, it suffices to assume $a$ to be $C^\alpha$ in $x$. We note that a partial result on the $L_p$ estimates has been obtained in \cite{MP14} when $a$ is H\"older in $x$ and $p$ is sufficiently large.

Another type of non-local equations which we would like to discuss here is equations with non-local time derivatives, e.g., the Caputo fractional time derivative, which has been used to model fractional diffusion in plasma turbulence. When $\alpha\in (0,1)$, it is given by
$$
\partial_t^\alpha u(t,x) = \frac{1}{\Gamma(1-\alpha)} \frac{d}{dt} \int_0^t (t-s)^{-\alpha} \left[ u(s,x) - u(0,x) \right] \, ds.
$$
Recently, there are many interesting work about parabolic equations with non-local time derivatives. For instance, De Giorgi--Nash--Moser type H\"{o}lder estimates for time fractional parabolic equations was obtained in Zacher \cite{MR3038123}, and for parabolic equations with fractional operators in both $t$ and $x$ in Allen et al. \cite{MR3488533}.

In Kim et al. \cite{MR3581300}, the authors proved the unique solvability in mixed $L_{p,q}$ spaces of the time fractional parabolic equation
\begin{equation*}
				%			\label{eq0525_01}
- \partial_t^\alpha u + a^{ij}D_{ij} u + b^i D_i u + c u = f
\end{equation*}
for $\alpha\in (0,2)$, under the assumption that the leading coefficients $a^{ij}$ are piecewise continuous in time and uniformly continuous in the spatial variables. Their proof is based upon a representation formula for solution to the time fractional heat operator
$$
-\partial_t^\alpha + \Delta
$$
using the Mittag--Leffler function
together with a perturbation argument.

In \cite{DKim19} this result was substantially extended in the case when $\alpha\in (0,1)$, where the leading coefficients $a^{ij}$ are allowed to be merely measurable in time and VMO in the space variables, as in Section \ref{sec2}. For the proof, the authors exploited the level set argument in \cite{MR1511579}. The main difficulty arises in the key step where one needs to estimate local $L_\infty$ estimates of the Hessian of solutions to locally homogeneous equations.
Starting from the $L_2$-estimate and applying the Sobolev type embedding results, one can only show that such Hessian are in $L_{p_1}$ for some $p_1>2$, instead of $L_\infty$.
Nevertheless, this allows us to obtain the $L_p$ estimate and solvability for any $p\in [2,p_1)$ and $a^{ij}=a^{ij}(t)$ by using a modified level set type argument. By repeating this procedure, one can iteratively increase the exponent $p$ for any $p\in [2,\infty)$. The case when $p\in (1,2)$ follows from a duality argument. For equations with the leading coefficients being measurable in $t$ and having small mean oscillations in $x$, it suffices to use a perturbation argument. This is achieved by incorporating the small mean oscillations of the coefficients into level set estimates of solutions having compact support in the spatial variables. After that, the standard partition of unity argument completes the proof.

In this direction, we also refer the reader to \cite{DKim20} for a corresponding result for divergence form equations with non-local time derivatives with leading coefficients measurable in a space variable and VMO with respect to other variables, as well as \cite{HKP19} for an extension of the result in \cite{MR3581300} to weighted mixed-norm Lebesgue spaces for the fractional heat equation
$$
-\partial_t^\alpha u+ \Delta u=f.
$$

Given the results in \cite{DK18} and \cite{MR3581300, HKP19}, it is natural to ask whether the result in \cite{DKim19} can be extended to the mixed-norm spaces and whether it is possible to also include weights. Unfortunately, it turns out that these extensions cannot be achieved by using the technique of iteration and the level set argument in \cite{DKim19}. In \cite{DKim20b}, the authors gave a definite answer to these two questions. In particular, it was proved that under the same assumptions on the coefficients as in \cite{DKim19}, for any $p,q\in (1,\infty)$ and Muckenhoupt weights $w_1(t) \in A_p(\bR), w_2(x) \in A_q(\bR^d)$, if $u$ satisfies
$$
- \partial_t^\alpha u + a^{ij}D_{ij} u + b^i D_i u + c u = f\quad \text{in}\ (0,T)\times \bR^d
$$
with the zero initial condition at $t=0$, then it holds that
$$
\||\partial_t^\alpha u|+|u|+|Du|+|D^2 u|\|_{L_{p,q,w}\left((0,T)\times\bR^d\right)} \leq N \|f\|_{L_{p,q,w}\left((0,T)\times\bR^d\right)},
$$
where $w=w_1(t)w_2(x)$ and $N$ is independent of $u$ and $f$.

For the proof, the authors adapted the mean oscillation argument mentioned above.
For this, one need to establish a H\"older estimate of $D^2 v$, where $v$ satisfies a certain homogeneous equation with coefficients depending only on $t$. Such an estimate can be obtained relatively easily for parabolic equations with the local time derivative term $v_t$ via somewhat standard local estimates. However, if the non-local time derivative is present, the local estimates do not work when improving the regularity of $v$ because the non-local time derivative of $v$ depends on all past states of $v$.
To overcome the difficulty from the non-local effect in time, the strategy is to consider an infinite cylinder $(-\infty,t_0)\times B_r(x_0)$ instead of the usual parabolic cylinder $Q_r(t_0,x_0)$. We write $u \in \bH_{p,0}^{\alpha,2}\left((0,T) \times \bR^d\right)$ if there exists a sequence $\{u_n\}$ such that $u_n \in C_0^\infty\left([0,T] \times \bR^d\right)$, $u_n(0,x) = 0$, and
\begin{equation*}
				%			\label{eq0316_01}
\|u_n - u\|_{p} + \|Du_n - Du\|_{p} + \|D^2u_n - D^2u\|_{p} + \|\partial_t^\alpha u_n - \partial_t^\alpha u\|_{p} \to 0
\end{equation*}
as $n \to \infty$, where $\|\cdot\|_{p} = \|\cdot\|_{L_{p}\left((0,T) \times \bR^d\right)}$. Let $\bH^{\alpha,2}_{p_0,0}((0,t_0)\times B_r)$ be the collection of functions in $\bH_{p,0}^{\alpha,2}\left((0,T) \times \bR^d\right)$ restricted in $(0,t_0)\times B_r$.

\begin{lemma}
                    \label{lem1}
Let $p_0\in (1,\infty)$, $t_0 \in (0,\infty)$, and $v\in \bH^{\alpha,2}_{p_0,0}((0,t_0)\times B_r)$ satisfy \begin{equation}
							\label{eq1.03}
- \partial_t^\alpha v + a^{ij}(t) D_{ij} v=0
\end{equation}
in $(0,t_0) \times B_r$, $r > 0$. Then there exists
$$
p_1 = p_1(d, \alpha,p_0)\in (p_0,\infty]
$$
satisfying
\begin{equation*}
				%			\label{eq0411_05}
p_1 > p_0 + \min\left\{\frac{2\alpha}{\alpha d + 2 - 2\alpha}, \alpha\right\}
\end{equation*}
such that
\begin{equation*}
				%			\label{eq0108_01}
\left( |D^2v|^{p_1} \right)_{Q_{r/2}(t_1,0)}^{1/p_1} \leq N
\sum_{j=1}^\infty j^{-(1+\alpha)} \left( |D^2v|^{p_0} \right)_{Q_r(t_1-(j-1)r^{2/\alpha},0)}^{1/p_0}
\end{equation*}
for any $t_1 \leq t_0$,
where $N=N(d,\delta, \alpha,p_0)$ and
$$
\left( |D^2v|^{p_1} \right)_{Q_{r/2}(t_1,0)}^{1/p_1} = \|D^2 v\|_{L_\infty\left(Q_{r/2}(t_1,0)\right)} \quad \text{if} \quad p_1 = \infty.
$$
If $p_0 > d/2+1/\alpha$, then
\begin{equation*}
				%			\label{eq0108_02}
[D^2 v]_{C^{\sigma \alpha/2,\sigma}(Q_{r/2}(t_1,0))} \leq N r^{-\sigma}
\sum_{j=1}^\infty j^{-(1+\alpha)} \left( |D^2v|^{p_0} \right)_{Q_r(t_1-(j-1)r^{2/\alpha},0)}^{1/p_0}
\end{equation*}
for any $t_1 \leq
t_0$, where $\sigma = \sigma(d,\alpha,p_0) \in (0,1)$.
Moreover, if $p_1 < \infty$, then $v \in \bH_{p_1,0}^{\alpha,2}\left((0,t_0) \times  B_{r/2}\right)$.
\end{lemma}

Lemma \ref{lem1} is proved by using a cutoff argument and the embedding theorems obtained in \cite{DKim19}. Then by an iteration argument and the simple inequality
$$
\sum_{j=1}^{k-1} j^{-(1+\alpha)}(k-j)^{-(1+\alpha)}\le N(\alpha) k^{-(1+\alpha)}
$$
for any $\alpha>0$, the following H\"older estimate is obtained.

\begin{proposition}
                            \label{prop5.1}
Let $1<p_0<p<\infty$, $t_0 \in (0,\infty)$, and $v\in \bH^{\alpha,2}_{p_0,0}((0,t_0)\times B_r)$ satisfy \eqref{eq1.03} in $(0,t_0) \times B_r$, $r > 0$. Then we have
\begin{align*}
			%				\label{eq2.11}
[D^2 v]_{C^{\sigma \alpha/2,\sigma}(Q_{r/2}(t_1,0))} \leq N r^{-\sigma}
\sum_{j=1}^\infty j^{-(1+\alpha)} \left( |D^2v|^{p_0} \right)_{Q_r(t_1-(j-1)r^{2/\alpha},0)}^{1/p_0}
\end{align*}
for any $t_1 \leq t_0$,
where $\sigma = \sigma(d,\alpha,p_0) \in (0,1)$ and $N=N(d,\delta,\alpha,p,p_0)$.
\end{proposition}

For solutions to inhomogeneous equations, the following bound is proved by taking the non-local effect of the operator into account. Denote $\cH^{\alpha,1}_{2,0}\left((0,T) \times B_r \right)$ to be the set of functions which can be approximated by a sequence $\{u_n\} \subset C^\infty_0\left([0,T] \times B_r \right)$ with $u_n(0,\cdot) = 0$ in the norm
$$
\|u\|_{\cH^{\alpha,1}_2\left((0,T) \times B_r \right)}
= \|u\|_{L_p\left((0,T) \times B_r \right)} + \|Du\|_{L_p\left((0,T) \times B_r \right)} + \|\partial_t^\alpha u\|_{\bH_p^{-1}\left((0,T) \times B_r \right)},
$$
where
\begin{align*}
&\|\partial_t^\alpha u\|_{\bH_p^{-1}\left((0,T) \times B_r \right)}\\
&:=\inf\{\|g\|_{L_p\left((0,T) \times B_r \right)}+\|f\|_{L_p\left((0,T) \times B_r \right)}\ :\ \partial_t^\alpha u=\Div g+f\ \text{in}\ (0,T) \times B_r\}.
\end{align*}
\begin{proposition}
                    \label{prop5.2}
Let $p_0\in (1,\infty)$, $f \in L_{p_0}\left( (0,t_0) \times B_1 \right)$, and $w\in \cH^{\alpha,1}_{2,0}((0,t_0)\times B_1)$ be a weak solution to
\begin{equation*}
				%			\label{eq1.01}
- \partial_t^\alpha w + a^{ij}(t) D_{ij} w= f(t,x)
\end{equation*}
 in $(0, t_0) \times B_1$ with the zero boundary condition on $\partial_p\left((0,t_0) \times B_1\right)$. Then we have
\begin{equation*}
                %                \label{eq3.02}
\left( |D^2 w|^{p_0} \right)_{Q_{1/2}(t_1,0)}^{1/p_0} \le \sum_{k=0}^\infty c_k(|f|^{p_0})^{1/p_0}_{(s_{k+1},s_k)\times B_{1}}
\end{equation*}
for any $t_1 \leq t_0$, where $s_k = t_1 - 2^k + 1$ and
\begin{equation*}
				%			\label{eq0212_03}
\sum_{k=0}^\infty c_k\le N=N(d,\delta,\alpha,p_0).
\end{equation*}
\end{proposition}

Define general cylinders
$$
Q_{r_1, r_2}(t,x) = (t-r_1^{2/\alpha}, t) \times B_{r_2}(x)
$$
and the strong maximal function
$$
\left(\cS \cM f\right) (t,x) = \sup_{Q_{r_1,r_2}(s,y) \ni (t,x)} \fint_{Q_{r_1,r_2}(s,y)} |f(r,z)| \, dz \, dr
$$
for any $f\in L_{1,\text{loc}}$. Clearly, $\cM f\le \cS \cM f$. It is well know that the Hardy--Littlewood theorem holds for the strong maximal function, i.e., for any $p\in (1,\infty)$,
$$
\|\cS \cM  f \|_{L_{p}(\bR^{d+1})} \le N \| f\|_{L_{p}(\bR^{d+1})}.
$$
Moreover, for any $p, q \in (1,\infty)$, $K_1 \in [1, \infty)$, $w_1(t) \in A_p(\bR,dt)$, $w_2(x) \in A_q(\bR^d, dx)$ satisfying
$$
[w_1]_{A_p} \leq K_1,\quad [w_2]_{A_q} \leq K_1,
$$
and $w(t,x) = w_1(t)w_2(x)$, we have for any $f \in L_{p,q,w}(\bR \times \bR^d)$,
\begin{equation}
                                        \label{eq8.54}
\|\cS\cM f\|_{L_{p,q,w}(\bR \times \bR^d)} \leq N \|f\|_{L_{p,q,w}(\bR \times \bR^d)},
\end{equation}
where $N = N(d,p,q,K_1) > 0$. This follows from \cite[Theorem 1.1]{B86} when $p=q$, and the extrapolation theorem of Rubio de Francia \cite{MR745140}.

Now by using a decomposition argument, it follows from Propositions \ref{prop5.1} and \ref{prop5.2} that for any $(t_0,x_0) \in (0,T] \times \bR^d$, $r \in (0,\infty)$, $\mu \in (0,1/4)$,
\begin{equation}
							\label{eq0219_01}
\begin{aligned}
&\left(|D^2u - (D^2u)_{Q_{\mu r}(t_0,x_0)}| \right)_{Q_{\mu r}(t_0,x_0)}
\leq N \mu^{\sigma} (\cS\cM |D^2u|^{p_0})^{\frac 1 {p_0}}(t_0,x_0)
\\
&\qquad\qquad + N \mu^{-\frac 1 {p_0}(d+\frac 2 \alpha)} \omega^{\frac 1 {\nu_2 p_0}}(R) (\cS\cM |D^2u|^{\nu_1 p_0})^{\frac 1 {\nu_1 p_0}}(t_0,x_0)
\\
&\qquad\qquad + N \mu^{-\frac 1 {p_0}(d+\frac 2 \alpha)} \left(\cS\cM |f|^{p_0}\right)^{\frac 1 {p_0}}(t_0,x_0)
\end{aligned}
\end{equation}
provided that $u \in \bH_{p_0}^{\alpha,2}\left((0,T) \times \bR^d\right)$ has compact support in $[0,T] \times B_{R}$ and satisfies
\begin{equation*}
				%			\label{eq0218_02}
-\partial_t^\alpha u + a^{ij}(t,x) D_{ij} u = f
\end{equation*}
in $(0,T) \times \bR^d$. Similar to \eqref{eq4.41b}, the inequality \eqref{eq0219_01} gives a pointwise sharp function estimate, which together with the weighted mixed-norm Fefferman--Stein theorem (Theorem \ref{thm4.2}) and \eqref{eq8.54} yield the desired estimate.

It is an interesting problem whether the results in \cite{DKim19, DKim20b} still hold in the factional wave case when $\alpha\in (1,2)$. In this range, the best result so far is in \cite{MR3581300}, where the coefficients are assumed to be piecewise continuous in time and uniformly continuous in the spatial variables. Would it be possible relax the condition to VMO coefficients, or even VMO$_x$ coefficients?

\section{Fully nonlinear equations with or without the convexity condition} \label{sec6}

In the last section of the paper, we discuss some recent progress in fully nonlinear elliptic and parabolic equations with discontinuous coefficients. The interior $W^2_p, p>d,$ estimates for a class of fully nonlinear uniformly elliptic equations of the form
$$
F(D^2u,x)=f(x)
$$
were first obtained by Caffarelli in \cite{Caf89}. See also \cite{CC95}. His proof is
geometric and is based on the Aleksandrov--Bakel'man--Pucci a
priori estimate, the Krylov--Safonov Harnack estimate, and
the ``crawling of ink spots'' lemma in \cite{KS80,Sa80} mentioned in the introduction.
Adapting this technique, similar estimates were proved by Wang \cite{Wa92} for
parabolic equations. By exploiting a weak reverse
H\"older's inequality, the result of \cite{Caf89} was sharpened by
Escauriaza in \cite{Es93}, who obtained the interior $W^2_p$-estimate for
the same equations allowing $p>d-\varepsilon$, with a small constant
$\varepsilon$ depending only on the ellipticity constant and $d$.
See also Winter \cite{Wi09} for an extension to the corresponding boundary estimate as well as the $W^2_p$-solvability of the boundary value problem. We also mention that a solvability theorem in the space $W^{1,2}_{p,\text{loc}}(Q)\cap C(\bar
Q)$ can be found in \cite{CKS00} for the boundary-value problem of fully
nonlinear parabolic equations. In these papers, a small oscillation
assumption in the integral sense is imposed on the operators. See, for
instance, \cite[Theorem 1]{Caf89}. However, as pointed out in
\cite[Remark 2.3]{Wi09} and in \cite{Kr10}
(see also \cite[Example 8.3]{CKS00} for a relevant discussion),
this assumption turns out to be equivalent to a small oscillation condition in the $L_\infty$ sense. %,
%which, particularly in the
%{\em linear\/} case, is the same as what is required in the
%classical $L_{p}$ theory based on the Calder\'on--Zygmund
%estimates.
Thus, the results in \cite{Caf89,Wa92,Es93,CKS00,Wi09} mentioned above are in general not
applicable to the operators under the VMO condition, in which local oscillations are measured in the average sense so that huge jumps in the $L_\infty$ norm are allowed.

In \cite{Kr10}, Krylov obtained interior $W^2_p,p>d,$ estimates for
elliptic Bellman's equations 
$$
\sup_{\omega\in \Omega}
[a^{ij} (\omega, x)D_{ij}u(x) + b_i
(\omega, x)D_iu(x)- (c(\omega, x) + \lambda)u(x) + f(\omega, x)] = 0
$$
in the whole space $\bR^d$ with VMO ``coefficients'', i.e.,
$$
\sup_{x\in \bR^d}\fint_{B_r(x)}\sup_{\omega\in\Omega}|a(\omega,y)-(a)_{B_r(x)}(\omega)|\ dy
$$
for a sufficiently small for all $r$ sufficiently small. 
This result was extended in \cite{DKL13} to the parabolic case, where both interior and boundary $W^{1,2}_p$ estimates for a class of fully nonlinear parabolic equations of the form
\begin{equation}
                                                \label{7.29.1}
\partial_t u(t,x) + F(D^{2}u(t,x),t,x)
+G(D^{2}u(t,x),D u(t,x),u(t,x),t,x)=0
\end{equation}
in cylindrical domains were proved and a solvability result for the terminal-boundary value problem was also obtained. Let $\bS$ be the set of symmetric $d\times d$ matrices,
fix some constants $\delta\in(0,1)$, and  $K\in\bR_{+}:=
(0,\infty)$. By $\bS_{\delta}$ we denote the subset
of $\bS$ consisting of matrices whose eigenvalues are
between $\delta$ and $\delta^{-1}$. Let $\cD\subset \bR^d$ be a bounded domain, $T\in (0,\infty)$, and $\cD_T=(0,T)\times \cD$.
The following condition is imposed in \cite{DKL13}.

\begin{assumption}
                        \label{assump6.1}
(i) $F(\sfu'',t,x)$ is convex and positive homogeneous of degree one
with respect to $\sfu''\in\bS$ and for all values of the arguments  and $\xi\in\bR^{d}$,
$$
 \delta|\xi|^2 \leq F(\sfu''+
\xi\xi^{*} ,t,x)
-F(\sfu'',t,x) \leq \delta^{-1}|\xi|^2.
$$

(ii)  $G(\sfu'',\sfu',\sfu,t,x)$, where $\sfu''\in\bS$, $\sfu'\in\bR^{d}$, and $\sfu\in\bR$, is nonincreasing
in $\sfu$ and  for all values of the arguments
$$
|G(\sfu'',\sfu',\sfu,t,x)-
G(\sfu'',\sfv',\sfv,t,x)|\leq K\big(|\sfu'-\sfv'|
+|\sfu-\sfv| \big),
$$
$$
|G(\sfu'',\sfu',\sfu,t,x)|\leq
\chi(|\sfu''|)|\sfu''|+ K(|\sfu'|+|\sfu|)+\bar{G}(t,x),
$$
where $\bar{G}$ and $\chi$ are given functions such that
$\chi:\bar\bR_+\to \bar \bR_+$ is bounded, decreasing, and
$\chi(r)\to0$ as
$r\to\infty$.

(iii) For all values of the arguments and
$\xi\in\bR^{d}$,
$$
\delta|\xi|^2 \leq F(\sfu''+\xi\xi^{*} ,t,x)
+G(\sfu''+\xi\xi^{*} ,\sfu',\sfu,t,x)
$$
$$
-F(\sfu'',t,x) -G(\sfu''  ,\sfu',\sfu ,t,x)\leq \delta^{-1}|\xi|^2 ;
$$

(iv) $F(\sfu'',t,x)+G(\sfu'',\sfu',\sfu,t,x)$ is convex with respect
to $\sfu''\in \bS$ and, for any $\sfu',\sfu,t,x$ at each
point $\sfu''$ at which $G(\sfu'',\sfu',\sfu,t,x)$ is differentiable
with respect to $\sfu''$, we have
\begin{equation*}
                %                   \label{3.25.2}
|G(\sfu'',\sfu',\sfu,t,x)-\sfu''_{ij}G_{\sfu''_{ij}}(\sfu'',\sfu',\sfu,t,x)|
\leq \hat{G}(t,x)M(|\sfu|)(1+|\sfu'|),
\end{equation*}
where $M(s)$ is a  continuous function on $\bR$
and $\hat{G}\in L_{1,\text{loc}}(\bR^{d+1})$.
\end{assumption}

\begin{remark}
The convexity of operators plays an important role in the regularity
theory of fully nonlinear elliptic and parabolic equations. For
elliptic equations without convexity assumptions, the best result
one can get is that viscosity solutions are in $C^{1+\alpha}$ (see
Trudinger \cite{Tr89}) under the condition that the operators are
sufficient regular (H\"older) with respect to the independent
variables. In fact, N. Nadirashvili and S. Vl\v{a}dut \cite{NV}
found an example which shows that even for elliptic operators
independent of the space variables viscosity solutions may not have
bounded second-order derivatives.
\end{remark}

\begin{remark}
A typical example when it is relatively easy to verify the hypotheses is given by the following Bellman's equation:
\begin{align}
                                                \label{7.14.2}
\partial_t u(t,x) + &\sup_{\omega\in \Omega} [a^{ij}(\omega,t,x)
 D_{ij} u(t,x)+b^{i }(\omega,t,x)D_{i}u(t,x)\notag\\
&-c(\omega,t,x)u(x)    + f(\omega, t, x)]=0,
\end{align}
  where the set $\Omega$ is a separable metric space, $a =(a^{ij} )$,
$b =(b^{i} )$, $c\geq0 $, and $f $
are given functions which are measurable in $(t,x)$ for
 each $\omega\in \Omega$
 and  continuous in
$\omega$ for each $(t,x)$. Introduce
\begin{align*}
&F(\sfu'',t,x)=\sup_{\omega\in \Omega}  a^{ij}(\omega,t,x)\sfu''_{ij},\\
&G(\sfu'',\sfu',\sfu,t,x)\\
&=\sup_{\omega\in \Omega} [a^{ij}(\omega,t,x)
\sfu''_{ij}+b^{i }(\omega,t,x)\sfu'_{i}
-c(\omega,t,x)\sfu    + f(\omega, t, x)]-F(\sfu'',t,x)
\end{align*}
and assume that for any $\omega$ the function $a^{ij}(\omega,t,x)\sfu''_{ij}$
satisfies Assumption \ref{assump6.1} (i) and the function
$b^{i }(\omega,t,x)\sfu'_{i}
-c(\omega,t,x)\sfu    + f(\omega, t, x)$ satisfies Assumption \ref{assump6.1} (ii).  Then
$F$  and $G$ satisfy Assumption \ref{assump6.1} (i)-(iii)  with
$\chi\equiv 0$.
\end{remark}

The following VMO assumption is imposed on the leading term in \eqref{7.29.1}
with a constant $\theta\in (0,1]$ to be specified later.

\begin{assumption}
                                \label{assump1}
There exists $R_0\in(0,1]$ such that for any
$r\in (0, R_0]$, $\tau\in \bR$, and  $z\in \cD$,
one can find a function $\bar{F} (\sfu'' )$ (independent
of $(t,x)$) satisfying condition Assumption \ref{assump6.1} (i) and such that
for any $\sfu''\in\bS$ with $|\sfu''|=1$ we have
\begin{equation}
                                                     \label{7.30.2}
\int_{Q_{r}(\tau,z)}
|
F(\sfu'' ,t,x)-\bar{F}(\sfu'')| \,dx\,dt\leq \theta
r^{ d+2}.
\end{equation}
\end{assumption}
For instance, \eqref{7.30.2} is satisfied by \eqref{7.14.2} if for
any
$r\in(0,R_{0}]$, $t\in\bR$, and $z\in \cD $ one can find
functions $\bar{a}^{ij}(\omega)$ such that the functions
$\bar{a}^{ij}(\omega)\sfu''_{ij}$ satisfy Assumption \ref{assump6.1} (i) and for any
$\sfu''\in\bS$ with $|\sfu''|=1$
$$
\int_{Q_{r}(\tau,z)}
|\sup_{\omega}a^{ij}(\omega,t,x)\sfu''_{ij}
-\sup_{\omega}\bar{a}^{ij}(\omega)\sfu''_{ij}|\,dx\,dt
\leq \theta
r^{ d+2}.
$$

One of the main results in \cite{DKL13} is the following solvability result.
\begin{theorem}[Theorem 1.1 of \cite{DKL13}]
%                                    \label{thm6.1}
Let $p>d+1$ be a constant,
$T\in\bR_{+}$, and let $\cD$ be a bounded
$C^{1,1}$ domain in
$\bR^d$.
Assume that $\bar G\in L_{p}(\cD_{T})$.
Then there exists a constant $\theta\in (0,1]$ depending
only on $d$, $p$, $\delta$, and the $C^{1,1}$ norm of $\partial \cD$
such that if Assumptions \ref{assump6.1} and \ref{assump1} are satisfied with this
 $\theta$, then the following assertions hold.
For any $g\in W^{1,2}_p(\cD_T)$, there is a unique solution
$u\in W^{1,2}_p(\cD_T)$ to \eqref{7.29.1} such
that $u=g$ on $\partial_p\cD_T$. Moreover, we have
\begin{equation*}
                %                \label{eq16.01}
\|u\|_{W^{1,2}_p(\cD_T)}\le N
\| \bar G\|_{L_p(\cD_T)}+N\|g\|_{W^{1,2}_p(\cD_T)}
+ N_{0} ,
\end{equation*}
where $N$ depends only on $d$, $p$, $\delta$, $K$,
$R_{0}$, the  $C^{1,1}$
norm of $\partial \cD$, and $\text{\rm diam}(\cD)$
 and $ N_{0} $ depends only on the same objects,  $T$,  and
$\chi$. In particular, $ N_{0} =0$ if $\chi\equiv0$.
\end{theorem}

We remark that in contrast to linear equations, to the best of our knowledge, there is no  $W^{1,2}_p$ estimate for  fully nonlinear equations with ``coefficients'' measurable in the time variable. For example, consider the equation
$$
u_t+F(D^2 u,t)=0,
$$
where $F$ satisfies the ellipticity and convexity conditions with respect to $D^2 u$. When $F$ is a measurable function of $t$, we do not even know whether $D^2 u$ is bounded (or $u\in C^{1,1}$ with respect to $x$).

The key ingredients of the proofs in \cite{Kr10,DKL13} are the Evans--Krylov $C^{2,\alpha}$ estimate applied to homogeneous equations with constant coefficients and a $W^2_\varepsilon$ estimate for elliptic equations with measurable coefficients, which is originally due to Lin \cite{Li86} and extended to the parabolic case in \cite{Kr10}.

\begin{lemma}
            %           \label{plin}
Let $r\in(0,\infty)$ and let
$u\in   C(\bar Q_{r})\cap  W^{1,2}_{d+1}(Q_{\rho})$
for any $\rho\in (0,r)$. Then there are constants $\gamma\in(0, 1]$ and $N$,
depending only on
$\delta, d$, such that
for any $L=a^{ij}D_{ij}$, where $a^{ij}$ are symmetric, bounded, measurable, and satisfy the ellipticity condition with ellipticity constant $\delta$, we have
$$
\fint_{Q_{r}} |D^2u|^\gamma \, dx \,  dt
 \leq  Nr^{-2\gamma}\sup_{\partial'Q_{r}}
 |u|^\gamma + N
\left(\fint_{Q_{r}}|\partial_{t}u +Lu|^{d+1} \,
dx\, dt\right)^{\gamma/{(d+1)}},
$$
where $\partial'Q_{r}$ is the union of the top and lateral boundaries of $Q_r$.
\end{lemma}

Observe that in the lemma above, we only have the bound of the $L_\gamma$ norm of $D^2 u$, where $\gamma$ could be small, so that the classical Fefferman--Stein theorem is not applicable. In this regard, another tool in the proof is a generalization of the Fefferman--Stein theorem, where the mean oscillation is measured in the $L_\gamma$ sense instead of the $L_1$ sense.

Later the convexity condition in the above theorem was relaxed in \cite{Kr13, DKr13, Kr13b, Kr18b}. In particular, in \cite{DKr13} the authors proved that for any uniformly parabolic fully nonlinear second-order equation
$$
\partial_t u(t,x)+H(u(t, x),D u(t,x),D^{2}u(t, x),t, x)=0
$$
with bounded measurable ``coefficients'' and bounded ``free'' term in  any cylindrical smooth domain  with  smooth boundary data, there is an approximating equation
$$
\partial_t u(t,x)+\max(H[u],P[u]-K)=0,
$$
which has a unique continuous solution with the first derivatives bounded and the second spacial derivatives locally bounded. Here $P[u]=P(D^2 u)$ is a convex elliptic operator independent of $(t,x)$ and the approximating equation is constructed by modifying the original one only for large values of $|u|$, $|Du|$, and $|D^2 u|$. The novelty of the results in \cite{Kr13, DKr13} is that no convexity assumption is imposed on the equation. The main idea is that on the set $\Gamma$, where the second-order derivatives of $u$ are large,  we have
$$
\partial_t u+P[u]=K.
$$ and by the maximum principle the second order derivative on $\Gamma$ are
controlled by their values on the boundary of $\Gamma$, where they are under
control by the definition of $\Gamma$. The implementation of this idea, however,
requires sufficient regularity of solutions. Since this is not
known a priori, the above idea was applied at the level of finite differences, which is quite involved.

The following solvability result for fully nonlinear parabolic equations
\begin{equation}
                                                \label{7.29.1b}
\partial_t u + F(D^{2}u,u,t,x)
+G(D^{2}u,D u,u,t,x)=0
\end{equation}
under a relaxed convexity condition can be found in \cite{Kr18b}. Let $\theta, \hat \theta\in (0,1)$ be constants to be specified.
\begin{assumption}
                                \label{assump6.5}
(i) For $\sfu''\in\bS$, $\sfu'\in\bR^{d}$, $\sfu\in\bR$, and $(t,x)\in \bR^{d+1}$,
$$
|G(\sfu'',\sfu',\sfu,t,x)\le \hat\theta|\sfu''|+K_0(|\sfu'|+|\sfu|)+\bar G(t,x),
$$
where $\bar G\ge 0$.

(ii) The function $F$ is Lipschitz continuous with respect to $\sfu''$
 with Lipschitz constant $K_{F}$ and  $F(0,t,x)\equiv0$.

(iii) There exist $R_0\in(0,1]$  and $\tau_{0}\in[0,\infty)$
such that, if $r\in (0, R_0]$, $z\in \bR^{d+1}$, $Q_r(z)\subset \cD_T$, and $\sfu\in \bR$, then one can find a convex function
$$
\bar{F} (\sfu'' )=
\bar{F}_{z,r,\sfu} (\sfu'' ),
$$
for which we have $\bar{F}(0)=0$ and
$ D_{\sfu'' }\bar{F} \in\bS_{\delta}$ at all points of differentiability
of $\bar{F}$. Moreover, for any $\sfu''\in\bS$ with $|\sfu''|=1$
($|\sfu''|:=  \text{tr } ^{1/2}(\sfu''\sfu'')$),
 we have
\begin{equation*}
                %                                \label{7.30.2}
\int_{Q_{r}(z)}\sup_{\tau>\tau_{0}}\tau^{-1}
\big|F\big(\tau \sfu'',\sfu, t,x\big)-\bar{F}(\tau \sfu'')\big| \,dx\,dt\leq
\theta \big|Q_{r}(z)\big|.
\end{equation*}
%where by $|A|$ we denote the volume of $A$ in $\bR^{d}$.

(iv) There exists a continuous increasing function $\omega_F(\tau),\tau\ge 0$, satisfying $\omega_F(0)=0$ and for any $\sfu,\sfv\in \bR$, $(t,x)\in \cD_T$, and $\sfu''\in \bS$, we have
$$
|F(\sfu'',\sfu,t,x)-F(\sfu'',\sfv,t,x)|\le \omega_F(|\sfu-\sfv|)|\sfu''|.
$$
\end{assumption}
Note that in (iii) the ``VMO'' condition on $F(\sfu'',u,t,x)$ with respect to $(t,x)$ is only imposed for $\sfu''$ sufficiently large. The following condition is also used for the solvability of equations. It is worth mentioning that in contrast to the previous results in the literature (see, for example, \cite{CKS00}), no Lipschitz continuity with respect to $\sfu'$ and $\sfu$ is assumed.

\begin{assumption}
                                \label{assump6.6}
For any $(t, x) \in \bR^{d+1}$, the function
$$
H(\sfu'',\sfu',\sfu,t,x):=F(\sfu'', t, x)+G(\sfu'',\sfu',\sfu,t,x)
$$
is continuous with respect to $(\sfu',\sfu)$, and Lipschitz continuous with respect to $\sfu''$, and at all points of differentiability of $H$ with respect to $\sfu''$, we have $D_{\sfu''}H \in\bS_{\delta}$.
\end{assumption}

For $\rho>0$, denote
$$
\cD^\rho=\{x\in \cD\,:\, B_r(x)\in \cD\},\quad
\cD_T^\rho=[0,T-\rho^2)\times \cD^\rho.
$$

\begin{theorem}[Theorem 12.1.10 of \cite{Kr18b}]
Let $p>d+2$ and $\cD$ be a bounded domain in $\bR^d$ which satisfies the exterior ball condition.
There exist constants $\theta,\hat\theta\in (0,1)$, depending only on $d$, $p$, $\delta$, $K_F$ such that if Assumptions \ref{assump6.5} and \ref{assump6.6} are satisfied, and
$\bar G\in L_p(\cD_T)$, then for any $g \in  C(((0,T)\times \partial\cD)\cup (\{T\}\times \bar \cD))$, there exists
$$
u\in \bigcup_{\rho>0}W^{1,2}_p(\cD_T^\rho)\cap C(\bar \cD_T)
$$
satisfying \eqref{7.29.1b} and $u=g$ on $((0,T)\times \partial\cD)\cup (\{T\}\times \bar \cD)$.
\end{theorem}

The exterior ball condition in the theorem above can be replaced with the weaker exterior measure condition:
$$
|B_r(x)\setminus \cD|\ge \delta |B_r(x)|
$$
for any sufficiently small $r>0$ and $x\in \partial\cD$. We refer the reader to Remark 12.1.12 of \cite{Kr18b} for a discussion about a similar solvability result in \cite{CKS00}, where viscosity solutions are considered and $H$ is assumed to be convex (or concave) in $\sfu''$ and continuous in $(t, x)$.

In \cite{DKr18}, the authors proved weighted and mixed-norm Sobolev estimates for fully nonlinear elliptic and parabolic equations in the whole space, under a relaxed convexity condition with almost VMO dependence on space-time variables similar to Assumption \ref{assump6.5}.  As a typical example, weights which are powers of the distance to the boundary were considered.
The corresponding interior and boundary estimates are also obtained. For the proof, the authors established the following local version of the Fefferman-Stein sharp functions theorem. Recall the notation in Section \ref{sec4}.
For $m\in \bZ$ and $\gamma\in(0,1]$, introduce
$$
u^{\#}_{\gamma,m}(x)=\sup_{n\geq m}
\sup_{\substack{Q^\alpha_n\in \bC_{n},\\
Q^\alpha_n\ni x}} \Big(\fint_{Q^\alpha_n}
 \fint_{Q^\alpha_n}|u(z)-u(y)|^{\gamma}\mu(dz)\mu(dy)\Big)^{1/\gamma},
$$
$$
\cM_{m}v=\sup_{n\leq m}v_{\mid n}.
$$
For $ \beta\in(0,1]$,  we say that $w$ is of $\beta$-type (the $A_\infty$ condition) if
$$
\frac{w(A)}{w(C)}\leq N_{w,\beta} \frac{|A|^{\beta}}{|C|^{\beta}}
$$
for any measurable $A\subset C$ and $C\in \bC_{\infty}:=\cup_n C_n$,
where $N_{w,\beta}$ is a (finite) constant independent of $C$ and $A$.
\begin{theorem}[Corollary 2.10 of \cite{DKr18}]
Let $p\in (1,\infty)$ and $m\in \bZ$. Assume that $|u|_{\mid n}\to 0$
as $n\to-\infty$, and let $w$ be of $\beta$-type.  Then for any
$p>\gamma\beta$,
$$
\int_{\Omega}|u|^{p}w\,\mu(dx)\leq NI^{(p-\gamma\beta)/p}
J^{\gamma\beta/p},
$$
where
\begin{align*}
I=\int_{\Omega}|\cM u|^{p}w\,\mu(dx),\quad
J=\int_{\Omega}\big(  u^{\#}_{\gamma,m}
+\cM^{
1/\gamma}_{m}(|u|^{\gamma})\big)^{p}w\,\mu(dx),
\end{align*}
and the constant $N$ depends only on
$K_1$, $K_2$, $N_{w,\beta}$, $p$, $\beta$, and $\gamma$.
\end{theorem}

This theorem allows them to derive estimates without relying on a partition of unity argument, which is not applicable to general fully nonlinear operators.

For further results in this direction, we refer the reader to \cite{Kr17, Kr18, PT16, Huang19} and the recent monograph \cite{Kr18b} on Sobolev and viscosity solutions for fully nonlinear equations. %For elliptic equations, such result was obtained in \cite{Kr13}.

\section*{Acknowledgments}
The author is grateful to Nicolai Krylov and Doyoon Kim for their helpful suggestions and comments on an earlier version of the paper.

\end{document}